%% file: Bi-interpretation_in_set_theory.tex
\documentclass{amsart}

\title{Bi-interpretation in weak set theories}

\author{Alfredo Roque Freire}
\address[Afredo Roque Freire]
{Professor of Logic, University of Bras\'ilia, Bras\'ilia}
\email{alfrfreire@gmail.com}
\urladdr{http://alfredoroquefreire.com}

\author{Joel David Hamkins}
\address[Joel David Hamkins]
{Professor of Logic, Oxford University \&\ Sir Peter Strawson Fellow, University College, Oxford}
         \email{joeldavid.hamkins@philosophy.ox.ac.uk}
         \urladdr{http://jdh.hamkins.org}

\thanks{This research was supported by grant 2017/21020-0, S\~ao Paulo Research Foundation (FAPESP). The research project grew out of the first author's PhD dissertation \cite{Freire2019:PhDthesis}, with related philosophical work in \cite{freire2018counts,freire2019ontological}. Commentary can be made about this article on the second author's blog at \href{http://jdh.hamkins.org/bi-interpretation-in-weak-set-theories}{http://jdh.hamkins.org/bi-interpretation-in-weak-set-theories}.}

\usepackage{latexsym,amsfonts,amsmath,amssymb,mathrsfs}
\usepackage{bbm}
\usepackage[hidelinks]{hyperref}
\usepackage{relsize}
\usepackage[rgb,dvipsnames]{xcolor}
\usepackage{tikz}
\usepackage{pgflibrarysnakes}
\usetikzlibrary{arrows,arrows.meta,petri,topaths,positioning,shapes,shapes.misc,patterns,calc,decorations.pathreplacing,hobby,snakes}
\usepackage{wrapfig} 
\usepackage{float}
\usepackage[utf8]{inputenc}
\RequirePackage{doi}

\usepackage{enumitem}

\include{MathMacrosJDH}
\renewcommand\emptyset{\varnothing}

\newcommand\ZM{{\rm ZM}}

\newcommand\drawM[1]{\draw[fill=yellow!25!white,fill opacity=.7] (-2,-2) rectangle (2,2) node[right,opacity=1] {#1};}

\newcommand\drawN[1]{\draw[fill=red!50!white,fill opacity=.6] (0,0) circle (2);
\draw (1.4,1.4) node[above right=-.5mm] {#1};}

\newcommand\drawMbi[3]{\draw[fill=yellow!25!white,fill opacity=.6] (-2,-2) rectangle (2,2)
node[right,opacity=1] {#1};
\draw[-{>[scale=#3]},>=Stealth,opacity=1,very thin,shorten >=.1mm] (-4.5,2.1) to[bend right=8] node[near start,above] {#2} (-2,0);
}
\newcommand\drawNbi[3]{\draw[fill=red!50!white,fill opacity=.5] (0,0) circle (2);
\draw (1.35,1.35) node[above right=-.2mm] {#1};
\draw[-{>[scale=#3]},>=Stealth,opacity=1,very thin,shorten >=.1mm] (6,-2.1) to[bend right=8] node[pos=.2,above] {#2} (2,0);
}

\newcommand{\sss}{\leftrightarrow}

\usepackage[rgb,dvipsnames]{xcolor}

\begin{document}

\begin{abstract}
In contrast to the robust mutual interpretability phenomenon in set theory, Ali Enayat proved that bi-interpretation is absent: distinct theories extending \ZF\ are never bi-interpretable and models of \ZF\ are bi-interpretable only when they are isomorphic. Nevertheless, for natural weaker set theories, we prove, including Zermelo-Fraenkel set theory $\ZFCm$ without power set and Zermelo set theory Z, there are nontrivial instances of bi-interpretation. Specifically, there are well-founded models of $\ZFCm$ that are bi-interpretable, but not isomorphic---even $\<H_{\omega_1},\in>$ and $\<H_{\omega_2},\in>$ can be bi-interpretable---and there are distinct bi-interpretable theories extending $\ZFCm$. Similarly, using a construction of Mathias, we prove that every model of \ZF\ is bi-interpretable with a model of Zermelo set theory in which the replacement axiom fails.
\end{abstract}

\maketitle

\section{Introduction}

Set theory exhibits a robust mutual interpretability phenomenon: in a given model of set theory, we can define diverse other interpreted models of set theory. In any model of Zermelo-Fraenkel \ZF\ set theory, for example, we can define an interpreted model of $\ZFC+\GCH$, via the constructible universe, as well as definable interpreted models of $\ZF+\neg\AC$, of $\ZFC+\MA+\neg\CH$, of $\ZFC+\mathfrak{b}<\mathfrak{d}$, and so on for infinitely many theories. For these latter theories, set theorists often use forcing to construct outer models of the given model; but nevertheless the Boolean ultrapower method provides definable interpreted models of these theories inside the original model (see theorem \ref{Theorem.Set-theories-are-mutually-interpretable}). Similarly, in models of \ZFC\ with large cardinals, one can define fine-structural canonical inner models with large cardinals and models of \ZF\ satisfying various determinacy principles, and vice versa. In this way, set theory exhibits an abundance of natural mutually interpretable theories.

Do these instances of mutual interpretation fulfill the more vigourous conception of bi-interpretation? Two models or theories are mutually interpretable, when merely each is interpreted in the other, whereas bi-interpretation requires that the interpretations are invertible in a sense after iteration, so that if one should interpret one model or theory in the other and then re-interpret the first theory inside that, then the resulting model should be definably isomorphic to the original universe (precise definitions in sections \S\ref{Section.Interpretability-in-models} and \S\ref{Section.Interpretability-in-theories}). The interpretations mentioned above are not bi-interpretations, for if we start in a model of $\ZFC+\neg\CH$ and then go to $L$ in order to interpret a model of $\ZFC+\GCH$, then we've already discarded too much set-theoretic information to expect that we could get a copy of our original model back by interpreting inside $L$. This problem is inherent, in light of the following theorem of Ali Enayat, showing that indeed there is no nontrivial bi-interpretation phenomenon to be found amongst the set-theoretic models and theories satisfying \ZF. In interpretation, one must inevitably discard set-theoretic information.\goodbreak

\begin{theorem}[Enayat \cite{enayat2017variations}]\
 \begin{enumerate}
   \item \ZF\ is tight: no two distinct theories extending \ZF\ are bi-interpretable.
   \item Indeed, \ZF\ is semantically tight: no two non-isomorphic models of \ZF\ are bi-interpretable.
   \item What is more, \ZF\ is solid: if $M$ and $N$ are mutually interpretable models of \ZF\ and the isomorphism of $M$ with its copy inside the interpreted copy of $N$ in $M$ is $M$-definable, then $M$ and $N$ are isomorphic.
 \end{enumerate}
\end{theorem}

We introduce the concept of \emph{semantic tightness} in this paper, since we find it very natural; Enayat had proved statements (1) and (3); we provide proofs in section~\S\ref{Section.Bi-interpretation-does-not-occur}. One should view solidity as a strengthening of semantic tightness---it amounts essentially to a one-sided version of semantic tightness, requiring the models to be isomorphic not only in instances of bi-interpretation, but also even in the case that only one of the two interpretation isomorphisms is definable, rather than both of them. Thus, statement (3) is a strengthening of statement (2). And statement (2) is an easy strengthening of statement (1), as we explain in corollary~\ref{Corollary.ZF-is-tight}.

The proofs of these theorems seem to use the full strength of \ZF, and Enayat had consequently inquired whether the solidity/tightness phenomenon somehow required the strength of \ZF\ set theory. In this paper,
we shall find support for that conjecture by establishing nontrivial instances of bi-interpretation in various natural weak set theories, including Zermelo-Fraenkel theory $\ZFCm$, without the power set axiom, and Zermelo set theory Z, without the replacement axiom.

\begin{maintheorems*}\
\begin{enumerate}
  \item $\ZFCm$ is not solid.
  \item $\ZFCm$ is not semantically tight, not even for well-founded models: there are well-founded models of $\ZFCm$ that are bi-interpretable, but not isomorphic.
  \item Indeed, it is relatively consistent with \ZFC\ that $\<H_{\omega_1},\in>$ and $\<H_{\omega_2},\in>$ are bi-interpretable and indeed bi-interpretation synonymous.
  \item $\ZFCm$ is not tight: there are distinct bi-interpretable extensions of $\ZFCm$.
  \item {\rm Z} is not semantically tight (and hence not solid): there are well-founded models of {\rm Z} that are bi-interpretable, but not isomorphic.
  \item Indeed, every model of \ZF\ is bi-interpretable with a transitive inner model of {\rm Z} in which the replacement axiom fails.
  \item {\rm Z} is not tight: there are distinct bi-interpretable extensions of Z.
\end{enumerate}
\end{maintheorems*}

These claims are made and proved in theorems \ref{Theorem.ZFCm-is-not-solid}, \ref{Theorem.Homega1-Homega_2-bi-interpretable}, \ref{Theorem.Homega1-Homega2-synonymy}, \ref{Theorem.ZFCm-is-not-tight}, and \ref{Theorem.Zermelo-not-solid-not-tight}. We shall in addition prove the following theorems on this theme:

\begin{enumerate}\sl\setcounter{enumi}{6}
  \item Nonisomorphic well-founded models of \ZF\ set theory are never mutually interpretable.
  \item The \Vaananen\ internal categoricity theorem does not hold for $\ZFCm$, not even for well-founded models.
\end{enumerate}

These are theorems \ref{Theorem.Well-founded-models-are-never-mutually-interpretable} and \ref{Theorem.ZFCm-double-relation-model}. Statement (8) concerns the existence of a model $\<M,\in,\barin>$ satisfying $\ZFCm(\in,\barin)$, meaning $\ZFCm$ in the common language with both predicates, using either $\in$ or $\barin$ as the membership relation, such that $\<M,\in>$ and $\<M,\barin>$ are not isomorphic.\goodbreak

\section{Interpretability in models}\label{Section.Interpretability-in-models}

Let us briefly review what interpretability means. The reader is likely familiar with the usual interpretation of the complex field $\<\C,+,\cdot,0,1,i>$ in the real field $\<\R,+,\cdot>$, where one represents the complex number $a+bi$ with the pair of real numbers $(a,b)$. The point is that complex number field operations are definable operations on these pairs in the real field. Conversely, it turns out that the real field is not actually interpreted in the complex field (see \cite{Hamkins.blog2020:Real-numbers-are-not-interpretable-in-the-complex-field} for an elementary account), but it is interpreted in the complex field equipped also with the complex conjugate operation $z\mapsto \bar z$, for in this case you can define the real line as the class of $z$ for which $z=\bar z$. The reader is also likely familiar with the usual interpretation of the rational field $\<\Q,+,\cdot>$ in the integer ring $\<\Z,+,\cdot>$, where one represents rational numbers as equivalence classes of integer pairs $(p,q)$, written more conveniently as fractions $\frac pq$, with $q\neq 0$, considered under the equivalence relation $\frac pq\equiv\frac rs\iff ps=qr$; one then defines the rational field structure by means of the familiar fractional arithmetic. Another example would be the structure of hereditarily finite sets $\<\HF,\in>$, which is interpreted in the standard model of arithmetic $\<\N,+,\cdot,0,1,<>$ by the Ackermann relation, the relation for which $n\mathrel{E} m$ just in case the $n$th binary digit of $m$ is $1$; this relation is definable in arithmetic and the reader may verify that $\<\HF,\in>\cong\<\N,E>$.\goodbreak

More generally, we interpret one model in another as follows.

\begin{definition}\label{Definition.Interpreted-in-for-models}\rm
A model $N=\<N,R,\ldots>$ in a first-order language $\mathcal{L}_N$ is \emph{interpreted} in another model $M$ in language $\mathcal{L}_M$, if for some finite $k$ there is an $M$-definable (possibly with parameters) class $N^*\of M^k$ of $k$-tuples and $M$-definable relations $R^{N^*}$ on $N^*$ for relation symbols $R$ in $\mathcal{L}_N$, as well as functions (defined via their graphs) and constants, and an $M$-definable equivalence relation $\simeq$ on $N^*$, which is a congruence with respect to those relations, for which there is an isomorphism $j$ from $N$ to the quotient structure $\langle N^*,R^{N^*},\ldots\rangle/\simeq$.
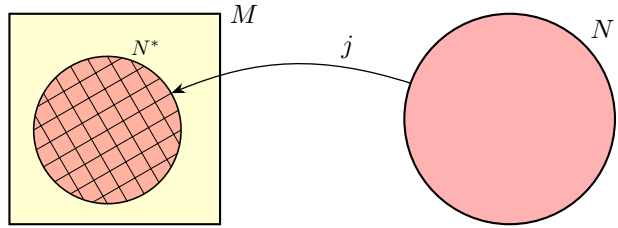
\begin{figure}[H]
\begin{tikzpicture}[scale=.7,thick]
\drawM{$M$}
\begin{scope}[scale=.7,shift={(-.2,-.3)},rotate=30,semithick,every node/.style={scale=.8}]
\drawN{$N^*$}
\node (Nb) at (2,0) {};
\clip (0,0) circle (2);
\draw[black,step=.5,very thin] (-2,-2) grid (2,2);
\end{scope}
\begin{scope}[shift={(7.5,0)},scale=1]
\drawN{$N$}
\node (Na) at (-1.88,.68) {};
\end{scope}
\draw[->,>=Stealth,bend right=20,line width=.5pt] (Na.center) to node[near start,above] {$j$} (Nb.center);
\end{tikzpicture}
\caption{Model $N$ is interpreted in model $M$}
\end{figure}
\end{definition}

In models supporting an internal encoding of finite tuples, such as the models of arithmetic or set theory, there is no need for $k$-tuples and one may regard $\Nbar\of M$ directly. In models of arithmetic or models of set theory with a definable global well-order, there is no need for the equivalence relation $\simeq$, since one can pick canonical least representatives from each class. Theorem \ref{Theorem.Interpretation-in-ZF-does-not-need-equivalence-relation} shows furthermore that in models of \ZF, even when there isn't a definable global well-order, one can still omit the need for the equivalence relation $\simeq$ by means of Scott's trick: replace every equivalence class with the set of its $\in$-minimal rank elements. Thus, in the model-theoretic terminology, the theory \ZF\ is said to \emph{eliminate imaginaries}. This trick doesn't necessarily work, however, in models of $\ZFCm$, without the power set axiom, since in this theory the class of minimal rank members could still be a proper class, although in \ZF\ it is always a set.

\begin{definition}\rm
Models $M$ and $N$ are \emph{mutually interpretable}, if each of them is interpreted in the other.
\begin{figure}[H]
\begin{tikzpicture}[scale=.7,thick]
\drawM{$M$}
\node (Ma) at (2,-1) {};
\begin{scope}[scale=.7,shift={(-.2,-.3)},rotate=30,semithick,every node/.style={scale=.8}]
\drawN{$N^*$}
\node (Nb) at (2,0) {};
\clip (0,0) circle (2);
\draw[black,step=.5,very thin] (-2,-2) grid (2,2);
\end{scope}
\begin{scope}[shift={(7.5,0)},scale=1]
\drawN{$N$}
\node (Na) at (-1.88,.68) {};
\begin{scope}[scale=.5,shift={(-.2,-.3)},rotate=30,semithick,every node/.style={scale=.8}]
\drawM{$M^*$}
\node (Mb) at (-2,0) {};
\draw[black,step=.667,very thin] (-2,-2) grid (2,2);
\end{scope}
\end{scope}
\draw[->,>=Stealth,bend right=20,line width=.5pt] (Na.center) to node[near start,above] {$j$} (Nb.center);
\draw[->,>=Stealth,bend right=20,line width=.5pt] (Ma.center) to node[near start,above] {$i$} (Mb.center);
\end{tikzpicture}
\caption{Models $M$ and $N$ are mutually interpreted}
\end{figure}
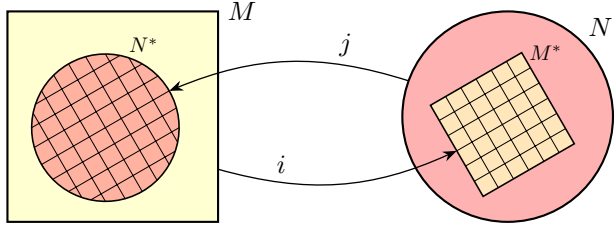
\end{definition}\goodbreak

Since (the quotient of) $N^*$ is isomorphic to $N$, we may find the corresponding copy of $M$ inside it, and similarly we may find the copy of $N$ inside $M^*$, as illustrated below, where we have now suppressed the representation of the equivalence relation.
\begin{figure}[H]
\begin{tikzpicture}[scale=.7,thick]
\drawM{$M$}
\node (Ma) at (2,-1) {};
\node (Mc) at (2,-.5) {};
\node (Mp) at (1.6,-.75) {};
\begin{scope}[scale=.7,shift={(-.2,-.3)},rotate=30,semithick,every node/.style={scale=.8}]
\drawN{$N^*$}
\node (Nb) at (2,0) {};
\begin{scope}[scale=.5,shift={(-.4,.2)},every node/.style={scale=.5},semithick=false]
\drawM{$\Mbar$}
\node (Md) at (2,-1) {};
\node (Mq) at (.5,-.5) {};
\end{scope}
\end{scope}
\begin{scope}[shift={(7.5,0)},scale=1]
\drawN{$N$}
\node (Na) at (135:2) {};
\node (Nc) at (160:2) {};
\node (Np) at (150:1.5) {};
\begin{scope}[scale=.5,shift={(-.2,-.3)},rotate=30,semithick,every node/.style={scale=.8}]
\drawM{$M^*$}
\node (Mb) at (-2,-1) {};
\begin{scope}[scale=.6,shift={(-.2,-.3)},every node/.style={scale=.5},semithick=false]
\drawN{$\Nbar$}
\node (Nd) at (170:2) {};
\node (Nq) at (135:.5) {};
\end{scope}
\end{scope}
\end{scope}
\draw[->,>=Stealth,out=150,in=30,line width=.5pt] (Na.center) to node[near start,above] {$j$} (Nb.center);
\draw[->,>=Stealth,bend right=20,line width=.5pt] (Ma.center) to node[near start,above] {$i$} (Mb.center);
\draw[-{[scale=.5]>},>=Stealth,out=10,in=20,looseness=3,very thin,shorten >=3pt] (Mp) node[scale=.7] {$a$} to node[pos=.4,right,scale=.6] {$ji$} (Mq)  node[scale=.6] {$\bar a$};
\draw[-{[scale=.5]>},>=Stealth,out=160,in=180,looseness=3.5,very thin,shorten >=2pt,shorten <=-1pt] (Np) node[scale=.7] {$u$} to node[pos=.4, left,scale=.6] {$ij$} (Nq)  node[scale=.6] {$\bar u$};
\end{tikzpicture}
\caption{Iterating the interpretations of models $M$ and $N$}
\end{figure}
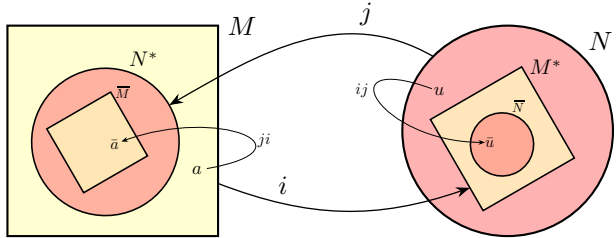

One may simply compose the isomorphisms $i$ and $j$ to form an isomorphism $ji$ between $M$ and $\Mbar$, which is now a definable subset of $M$ (although the map $ji$ may not be definable). Because in the general case these are quotient structures, one should think of the map $ji$ as a relation rather than literally a function, since each object $a$ in $M$ is in effect associated with an entire equivalence class $\bar a$ in the quotient structure of $\Mbar$. Similarly, the composition $ij$ is an isomorphism of the structure $N$ with $\Nbar$, which is now definable in $N$.

It will be instructive to notice that we have little reason in general to expect the map $ji$ to be definable in $M$, and this issue is the key difference between the mutual interpretation of models and their bi-interpretation.

\begin{definition}\rm
Models $M$ and $N$ are \emph{bi-interpretable}, if they are mutually interpretable in such a way that the isomorphisms $ji:M\cong\Mbar$ and $ij:N\cong\Nbar$ arising by composing the interpretation maps are definable in the original models $M$ and $N$, respectively.
\begin{figure}[H]
\begin{tikzpicture}[scale=.7,thick]
\drawM{$M$}
\node (Ma) at (2,-1) {};
\node (Mc) at (2,-.5) {};
\node (Mp) at (1.6,-.75) {};
\begin{scope}[scale=.7,shift={(-.2,-.3)},rotate=30,semithick,every node/.style={scale=.8}]
\drawN{$N^*$}
\node (Nb) at (2,0) {};
\begin{scope}[scale=.5,shift={(-.4,.2)},every node/.style={scale=.5},semithick=false]
\drawM{$\Mbar$}
\node (Md) at (2,-1) {};
\node (Mq) at (.5,-.5) {};
\end{scope}
\end{scope}
\begin{scope}[shift={(7.5,0)},scale=1]
\drawN{$N$}
\node (Na) at (135:2) {};
\node (Nc) at (160:2) {};
\node (Np) at (135:1.65) {};
\begin{scope}[scale=.5,shift={(-.2,-.3)},rotate=30,semithick,every node/.style={scale=.8}]
\drawM{$M^*$}
\node (Mb) at (-2,-1) {};
\begin{scope}[scale=.6,shift={(-.2,-.3)},every node/.style={scale=.5},semithick=false]
\drawN{$\Nbar$}
\node (Nd) at (170:2) {};
\node (Nq) at (135:.4) {};
\end{scope}
\end{scope}
\end{scope}
\draw[->,>=Stealth,bend right, line width=.5pt] (Na.center) to node[near start,above] {$j$} (Nb.center);
\draw[->,>=Stealth,bend right=20,line width=.5pt] (Ma.center) to node[near start,above] {$i$} (Mb.center);
\draw[-{[scale=.5]>},>=Stealth,bend right,very thin,shorten >=2pt,shorten <=-1pt] (Mp) node[scale=.7] {$a$} to node[midway, above,scale=.5] {$ji$} (Mq)  node[scale=.6] {$\bar a$};
\draw[-{[scale=.5]>},>=Stealth,bend right,very thin,shorten >=2pt,shorten <=-1pt] (Np) node[scale=.7] {$u$} to node[pos=.1, right,scale=.5] {$ij$} (Nq)  node[scale=.6] {$\bar u$};
\end{tikzpicture}
\caption{Models $M$ and $N$ are bi-interpreted}
\label{Figure.Bi-interpretation}
\end{figure}
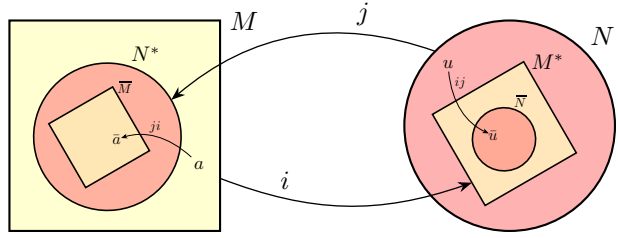
\end{definition}\goodbreak

A somewhat cleaner picture emerges, as indicated below, if we should simply identify the model $N$ with its isomorphic copy $N^*$ and $M^*$ with $\Mbar$. Let us now use $i$ to denote the isomorphism of $M$ with its copy $\Mbar$ inside $N$, and $j$ is the isomorphism of $N$ with its copy $\Nbar$ inside $\Mbar$. This picture applies with either mutual interpretation or bi-interpretation, the difference being that in the case of bi-interpretation, the isomorphisms $i$ and $j$ are definable in $M$ and $N$, respectively. Furthermore, by iterating these interpretations, in the case either of mutual interpretation or bi-interpretation, one achieves a fractal-like nested sequence of definable structures, with isomorphisms at each level. In the case of bi-interpretation, all these maps are definable in $M$, and any of the later maps is definable inside any of the successive copies of $M$ or $N$ in which it resides.

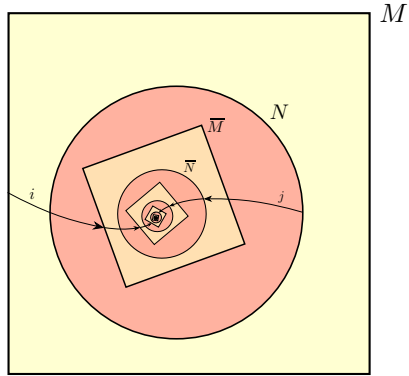
\begin{figure}[H]
\begin{tikzpicture}[scale=1.2,thick]
\begin{scope}[thick,shift={(7.5,0)}]
\drawM{$M$}
\begin{scope}[scale=.7,shift={(-.2,-.3)},every node/.style={scale=.8},semithick]
\drawN{$N$}
\begin{scope}[scale=.5,shift={(-.4,.2)},rotate=20,every node/.style={scale=.6},line width=.2mm]
\drawMbi{$\Mbar$}{$i$}{1}
\begin{scope}[scale=.7,shift={(-.2,-.3)},every node/.style={scale=.5},thin]
\drawNbi{$\Nbar$}{$j$}{.7}
\begin{scope}[scale=.5,shift={(-.4,.2)},rotate=20,very thin]
\drawMbi{}{}{.5}
\begin{scope}[scale=.7,shift={(-.2,-.3)},ultra thin]
\drawNbi{}{}{.4}
\begin{scope}[scale=.5,shift={(-.4,.2)},rotate=20,line width=.02mm]
\drawMbi{}{}{.3}
\begin{scope}[scale=.7,shift={(-.2,-.3)},-{[scale=.1]>}]
\drawNbi{}{}{.25}
\begin{scope}[scale=.5,shift={(-.4,.2)},rotate=20,line width=.01mm]
\drawM{}{}
\begin{scope}[scale=.7,shift={(-.2,-.3)}]
\drawN{}
\begin{scope}[scale=.5,shift={(-.4,.2)},rotate=20,line width=.005mm]
\drawM{}{}
\begin{scope}[scale=.7,shift={(-.2,-.3)}]
\drawN{}
\end{scope}
\end{scope}
\end{scope}
\end{scope}
\end{scope}
\end{scope}
\end{scope}
\end{scope}
\end{scope}
\end{scope}
\end{scope}
\end{scope}
\end{tikzpicture}
\caption{Iterated bi-interpretations}\label{Figure.Iterated-interpretation}
\end{figure}

An even stronger connection between models is the notion of bi-interpretation synonymy. This relation was introduced by Bouvere in \cite{Bouvere1965} and recently deepened by Friedman and Visser in \cite{friedman2014bi}.

\begin{definition}\rm
Models $M$ and $N$ are \emph{bi-interpretation synonymous}, if there is a bi-interpretation for which (i) the domains of the interpreted structures are in each case the whole structure; and (ii) the equivalence relations used in the interpretation are the identity relations on $M$ and $N$.
\end{definition}

In remarks after theorem \ref{Theorem.Vaananen}, we explain how every instance of bi-interpretation between models of \ZF\ can be transformed to an instance of bi-interpretation synonymy.

\section{Interpretability in theories}\label{Section.Interpretability-in-theories}

Let us now explain how interpretability works in theories, as opposed to models. At bottom, an interpretation of one theory in another amounts to a uniform method of interpreting a model of the first theory in any model of the second theory. Specifically, we interpret one theory $T_1$ in another theory $T_2$ by providing a way to translate the $T_1$ language and structure into the language and structure available in $T_2$, in such a way that the translation of every $T_1$ theorem is provable in $T_2$.

In somewhat more detail, the interpretation $I$ of theory $T_1$ in language $\mathcal{L}_1$ into theory $T_2$ in language $\mathcal{L}_2$ should provide, first of all, an $\mathcal{L}_2$-formula $U(\bar x)$ that will define the interpreted domain of $k$-tuples $\bar x=(x_1,\dots,x_k)$ and an $\mathcal{L}_2$-expressible relation $\bar x =^I\bar y$, the interpretation of equality, that $T_2$ proves is an equivalence relation on tuples in $U$. Next, for each relation symbol $R$ of $\mathcal{L}_1$, the interpretation should provide a translation $R^I$ as an $\mathcal{L}_2$ formula of the same arity on $U$ as $R$ has in $\mathcal{L}_1$, and $T_2$ should prove that this relation is well-defined modulo $=^I$. Finally, functions are to be handled by considering their graphs $f(\bar x)=y$ as definable relations and interpreting these graph relations, but with the proviso that $T_2$ proves that the interpreted relation is indeed well-defined and functional on $U$ modulo $=^I$. Ultimately, therefore, the theory $T_2$ will prove that $=^I$ is a congruence with respect to the interpreted relations $R^I$ and functions $f^I$.\goodbreak

Having thus interpreted the atomic $\mathcal{L}_1$ structure, one may naturally extend the interpretation to all $\mathcal{L}_1$ assertions as follows.
\begin{itemize}
  \item $(\varphi\wedge\psi)^I=\varphi^I\wedge\psi^I$
  \item $(\neg\varphi)^I=\neg\varphi^I$
  \item $(\exists x\, \varphi(x))^I=\exists\bar x\ U(\bar x)\wedge \varphi^I(\bar x)$
\end{itemize}
Finally, for this to be an interpretation of $T_1$ in $T_2$, we insist that $T_2$ prove that $T_1$ holds for the interpreted structure, or in other words, that
 $$T_1\proves\varphi\qquad\text{ implies }\qquad T_2\proves \varphi^I.$$
Because the interpretation preserves the boolean and quantifier structure of interpreted formulas, it suffices that $T_2$ should prove the interpretation $\varphi^I$ of each axiom $\varphi$ of $T_1$. It is clear that if $T_1$ is interpreted in $T_2$ by interpretation $I$, then in any model $M\satisfies T_2$ we may use the interpretation to define a model $N=\<U,R^I,f^I,\ldots>/=^I$ of $T_1$. The domain of $N$ will consist of the quotient of $U$ by $=^I$ as $M$ sees it, and the relations $R$ and functions $f$ of $N$, if any, will be exactly those defined by $R^I$ and $f^I$ via the interpretation.

\begin{definition}\rm\
\begin{enumerate}
 \item Two theories $T_1$ and $T_2$ are \emph{mutually interpretable}, if each of them is interpretable in the other.
 \item Two theories $T_1$ and $T_2$ are \emph{bi-interpretable}, in contrast, if they are mutually interpretable by interpretations $I$ and $J$ respectively, which are provably invertible, in that the theory $T_1$ proves that the universe is isomorphic, by a definable isomorphism map, to the model resulting by first interpreting via $J$ to get a model of $T_2$ and then interpreting via $I$ to get a model of $T_1$ inside that model; and similarly the theory $T_2$ proves that its universe is definably isomorphic to the model obtained by first interpreting via $I$ to get a model of $T_1$ and then inside that model via $J$ to get a model of $T_2$.
\end{enumerate}
\end{definition}

With mutual interpretation of theories, one can start with a model $M$ of $T_1$ and then use $J$ to define within it a model $N=J^M$ of $T_2$, which can then define a model $M'=I^N$ of $T_1$ again. With mutual interpretability, there is no guarantee that $M'$ and $M$ are isomorphic or even elementarily equivalent, beyond the basic requirement that both are models of $T_1$. The much stronger requirements of bi-interpretability, in contrast, ensure that $M'$ and $M$ are isomorphic, and furthermore isomorphic by a definable isomorphism relation, which the theory $T_1$ proves is an isomorphism. And similarly when interpreting in models of $T_2$. Thus, with bi-interpretation, each theory sees how its own structure is faithfully copied via the definable isomorphisms $f_1$ and $f_2$ under iterations of the interpretations. With bi-interpretability, it therefore follows that \begin{enumerate}
  \item For every $\mathcal{L}_1$ assertion $\alpha$
  \begin{equation*}
  T_1 \vDash \forall\bar x\ \alpha(x_1,\ldots,x_n) \sss {\alpha^J}^I(f_1(x_1), f_1(x_2), \ldots, f_1(x_n)).
  \end{equation*}
  \item For every $\mathcal{L}_2$ assertion $\beta$
  \begin{equation*}
  T_2 \vDash \forall\bar y\ \beta(y_1,\ldots,y_m) \sss {\beta^I}^J(f_2(y_1), f_2(y_2), \ldots, f_2(y_m)).
  \end{equation*}
\end{enumerate}
In light of the quotients by the equivalence relations, the functions $f_1$ and $f_2$ are more properly thought of as relations $f_1(x)=\bar y$ well-defined with respect to those relations; they need not be functional on points, but only in the quotient.

\section{Mutual interpretation of diverse set theories}

Let us briefly establish the mutual interpretability phenomenon in set theory.

\begin{theorem}The following theories are pairwise mutually interpretable.\label{Theorem.Set-theories-are-mutually-interpretable}
  \begin{enumerate}
    \item \ZF
    \item \ZFC
    \item $\ZFC+\GCH$
    \item $\ZFC+V=L$
    \item $\ZF+\neg\AC$
    \item $\ZFC+\neg\CH$
    \item $\ZFC+\MA+\neg\CH$
    \item $\ZFC+\mathfrak{b}<\mathfrak{d}$
    \item Any extension of \ZF\ provably holding in a definable inner model or provably forceable by set forcing over such an inner model.
  \end{enumerate}
\end{theorem}

\begin{proof}
In many of these instances, the interpretation is obtained by defining a suitable inner model of the desired theory. For example, in any model of \ZF, we can define the constructible universe, and thereby find an interpretation of $\ZFC+V=L$ and hence $\ZFC+\GCH$ and so forth in \ZF. Some of the interpretations involve forcing, however, which we usually conceive as a method for constructing outer models, rather than defining models inside the original model. Nevertheless, one can use forcing via the Boolean ultrapower method to define interpreted models of the forced theory inside the original model. Let us illustrate the general method by explaining how to interpret the theory $\ZFC+\neg\CH$ in \ZF. In \ZF, we may define $L$, and thereby define a model of \ZFC\ with a definable global well-order. Inside that model, consider the forcing notion $\Add(\omega,\omega_2)$ to add $\omega_2$ many Cohen reals. Let $\B$ be the Boolean completion of this forcing notion in $L$, and let $U\of\B$ be the $L$-least ultrafilter on this Boolean algebra. Using the Boolean ultrapower method, we define the quotient structure $L^\B/U$ on the class of $\B$-names by the relations:
\begin{eqnarray*}
  \sigma=_U\tau &\Iff& \boolval{\sigma=\tau}\in U;\\
  \sigma\in_U\tau &\Iff& \boolval{\sigma\in\tau}\in U.
\end{eqnarray*}
The model $L^\B/U$ consists of the $=_U$ equivalence classes of $\B$-names in $L$ under the membership relation $\in_U$. The Boolean ultrapower \Los\ theorem shows that $L^\B/U$ satisfies every statement $\varphi$ whose Boolean value $\boolval{\varphi}$ is in $U$, and so this is a model of $\ZFC+\neg\CH$. We should like to emphasize that there is no need in the Boolean ultrapower construction for the ultrafilter $U$ to be generic in any sense, and $U\in L$ is completely fine (see extensive explanation in \cite[\S2]{HamkinsSeabold:BooleanUltrapowers}). By using the $L$-least such ultrafilter $U$ on $\B$, therefore, we eliminate the need for parameters---the Boolean ultrapower quotient $L^\B/U$ is a definable interpreted model of $\ZFC+\neg\CH$ inside the original model, as desired.

The same method works generally. Any forceable theory will hold in an interpreted model, using the forcing notion and the ultrafilter $U$ in the original model as parameters; when forcing over $L$ or $\HOD$ one may use the corresponding definable well-order to eliminate the need for parameters, as we did above. Any theory that is provably forceable over a definable inner model will be forceable over $L$ and therefore amenable to this parameter-elimination method.
\end{proof}

Apter, Gitman and Hamkins \cite{ApterGitmanHamkins2012:InnerModelsWithLargeCardinals} similarly establish in numerous instances that one can find transitive inner models $\<M,\in>$ of certain large cardinal theories usually obtained by forcing. For example, if there is a supercompact cardinal, then there is a definable transitive inner model with a Laver-indestructible supercompact cardinal, and another inner model with a non-indestructible supercompact cardinal, and another with a strongly compact cardinal that was also the least measurable cardinal, and so on.

As we emphasized in the introduction of this article, theorem \ref{Theorem.Set-theories-are-mutually-interpretable} is concerned with \emph{mutual} interpretation rather than \emph{bi-}interpretation. The interpretation methods used in this theorem are not invertible, and when following an interpretation, one cannot in general get back to the original model. Every interpretation of one set theory in another will inevitably involve a loss of set-theoretic information in the models in which it is undertaken. The main lesson of Enayat's theorem (theorem \ref{ali-theorem}) is that this problem is unavoidable: in fact, none of the theories above are bi-interpretable and there is no nontrivial bi-interpretation phenomenon to be found in set theory at the strength of \ZF.

One might have hoped to invert the forcing extension interpretations by means of the ground-model definability theorem (\cite{Laver2007:CertainVeryLargeCardinalsNotCreated}, see also \cite{FuchsHamkinsReitz2015:Set-theoreticGeology}), which asserts that every ground model $M$ is definable (using parameters) in its forcing extensions $M[G]$ by set forcing. That is, since we have interpreted the theory of the forcing extension, can't we reinterpret the ground model by means of the ground-model definability theorem? The problem with this idea is that the forcing Boolean ultrapower model $M^{\B}/U$ that we used in theorem \ref{Theorem.Set-theories-are-mutually-interpretable} is not actually a forcing extension of $M$, but rather a forcing extension of its own ground model $\check M/U$, which in general is not the same as $M$, although it is an elementary extension of $M$ by the Boolean ultrapower map $i:M\to \check M/U$. This map is an isomorphism only when $U$ is $M$-generic, which is not true here for nontrivial forcing because $U\in M$. Indeed, the Boolean ultrapower model $\check M/U$ is usually ill-founded, even when $M$ is well founded (although there can be instances of well-founded Boolean ultrapowers connected with very large large cardinals; see \cite{HamkinsSeabold:BooleanUltrapowers}). Nevertheless, the method does allow us to mutually interpret the theory of $M$ with the theory of what it would be like in a forcing extension of $M$ after forcing with $\B$, as explained by the following theorem:

\begin{theorem}
For any model $M\satisfies\ZFC$ and any notion of forcing $\B\in M$, the theory of $M$ with constants for every element of $M$ is mutually interpretable with the theory, in the forcing language with constants for every $\B$-name in $M$, asserting that the universe is a forcing extension via $\B$ of a ground model with the theory of $M$.
\end{theorem}

\begin{proof}
The latter theory, describing what it would be like in a forcing extension of $M$ by $\B$, is the same theory as used in the naturalist account of forcing in \cite{Hamkins2012:TheSet-TheoreticalMultiverse}. The theory has a predicate symbol $\check M$ for an inner model and asserts that this satisfies the same theory as $M$ and that the universe is a forcing extension $\check M[G]$ for some $\check M$-generic filter $G\of\check\B$. To be clear, we are not saying that the model $M$ is mutually interpretable with some actual forcing extension model $M[G]$, but rather only that the theory of $M$ is mutually interpretable with the theory expressing what it is like in such a forcing extension. We allow parameters in the interpretation definitions, although in many cases these are eliminable. Assume that $\B$ is a complete Boolean algebra in $M$, and let $U\of\B$ be an ultrafilter in $M$. Inside $M$, we may define the forcing Boolean ultrapower $M^{\B}/U$ as described above, and this is a model of the theory asserting (in the forcing language) that the universe is a forcing extension of $\check M$ by $\check \B$. That this is true is part of the theory of $M$, and so this interpretation works inside any model of the theory of $M$. Conversely, inside any model of the latter theory, we can define the ground model $\check M$, which will be a model of the theory of $M$.
\end{proof}

\section{Further background on interpreting in models of set theory}\label{Section.Background-on-interpretating-in-set-theory}

Before getting to the main results, let us prove some further useful background material on interpretations in models of set theory. We have already mentioned in the remarks after definition \ref{Definition.Interpreted-in-for-models} that in models of set theory, we do not need to use $k$-tuples in interpretation, since we have definable pairing functions. But also, it turns out that when interpreting in a model of \ZF, we do not need the equivalence relation:

\begin{theorem}\label{Theorem.Interpretation-in-ZF-does-not-need-equivalence-relation}
If a structure $A$ is interpreted in a model of \ZF\ set theory, then there is an interpretation in which the equivalence relation is the identity. \end{theorem}

This theorem appears as exercise 2 of section 4.4 of \cite{Hodges1993:ModelTheory}.

\begin{proof}
The argument amounts to what is known as ``Scott's trick.'' If $\simeq$ is a definable equivalence relation in a model $M\satisfies\ZF$, then let us replace every equivalence class $[a]_\simeq$ with the Scott class $[a]^*_\simeq$, which consists of the elements of the equivalence class having minimal possible $\in$-rank in that class. This is a definable set in $M$, and from the Scott class one can identity the original $\simeq$ equivalence class. So we can replace the original interpretation modulo $\simeq$ with the induced interpretation defined on the Scott classes. And since two Scott classes correspond to the same interpreted object if and only if they are equal, we have thereby eliminated the need for the equivalence relation.
\end{proof}

This proof does not work in $\ZFCm$, that is, in set theory without the power set axiom, because the minimal rank instances from a class may still form a proper class, and so Scott's trick doesn't succeed in reducing the class to a set. And so we had asked the question, posting it on MathOverflow \cite{Hamkins2020.MO350542:Can-Homega1-Homega2-be-in-synonymy?}:

\begin{question}\label{Question.ER-required?}
Is there a structure that is interpretable in a model of $\ZFCm$, but only by means of a nontrivial equivalence relation?
\end{question}

The question was answered affirmative by Gabe Goldberg \cite{Goldberg2020.MO350585:Can-Homega1-Homega2-be-in-synonymy?}, assuming the consistency of large cardinals. He pointed out that if $\AD^{L(\R)}$ holds and $\delta^1_2=\omega_2$, there is a projectively definable prewellordering on the reals of order type $\omega_2$, and so the order structure $\<\omega_2,<>$ is interpretable in $\<H_{\omega_1},\in>$, but there is no definable order in this structure of type $\omega_2$, because there is, he proves, no injection of $\omega_2$ into $H_{\omega_1}$ in $L(\R)$.  Therefore, under these assumptions, we have a structure interpretable in a model of $\ZFCm$, using a nontrivial equivalence relation, but not without. Can one produce an example in \ZFC?

Next, we establish the phenomenon of automatic bi-interpretability for well-founded models of set theory.\goodbreak

\begin{theorem}\label{Theorem.Well-founded-model-interpretation-is-definable}
If a well-founded model $M$ of $\ZF^-$ is interpreted in itself via $i:M\to\Mbar/\simeq$, then the interpretation isomorphism map $i$ is unique and furthermore, it is definable in $M$.
\end{theorem}

\begin{proof}
Assume that $M$ is a well-founded model of $\ZF^-$, which we may assume without loss to be transitive, and that $\<M,\in>$ is interpreted in itself via the interpretation map $i:\<M,\in>\cong\<\Mbar,\barin>/\simeq$, where $\Mbar\of M$ is a definable class in $M$ with a definable relation $\barin$, and $\simeq$ is a definable equivalence relation on $\Mbar$, which is a congruence with respect to $\barin$.
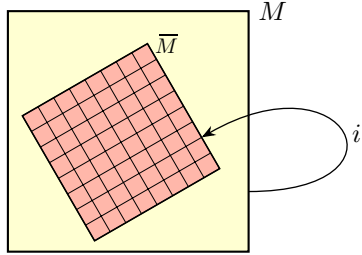
$$\begin{tikzpicture}[scale=.8,thick]
\drawM{$M$}
\node (Ma) at (2,-1) {};
\begin{scope}[scale=.6,fill=orange,shift={(-.2,-.3)},rotate=30,semithick,every node/.style={scale=.8}]
\draw[fill=red!40!white,fill opacity=.7] (-2,-2) rectangle (2,2) node[right,opacity=1] {$\Mbar$};
\node (Mb) at (2,-1) {};
\draw[black,step=.5,very thin] (-2,-2) grid (2,2);
\end{scope}
\draw[->,>=Stealth,out=0,in=30,looseness=6,line width=.5pt] (Ma.center) to node[pos=.5,right] {$i$} (Mb.center);
\end{tikzpicture}$$
Since models of set theory admit pair functions, however, we may assume directly that $\Mbar\of M$ rather than $\Mbar\of M^k$ with $k$-tuples.

Let $\pi=i^{-1}:\Mbar\to M$ be the inverse map, viewed as a map on $\Mbar$. Since $i$ is an isomorphism, it follows that $b\in a\Iff i(b)\mathrel{\barin}i(a)$, and consequently, $\pi(\bar a)=\set{\pi(\bar b)\mid \bar b\mathrel{\barin}\bar a}$. Thus, $\pi$ is precisely the Mostowski collapse of the relation $\barin$ on $\Mbar$, which is well-founded precisely because it has a quotient that is isomorphic to $\<M,\in>$. Thus, the map $i$ is unique, for it is precisely the inverse of the Mostowski collapse of $\<\Mbar,\barin>/\simeq$, as computed externally to $M$ in the context where $i$ exists.

What remains is to show that $M$ itself can undertake this Mostowski collapse. While the $\simeq$-quotient of $\barin$ is a well-founded extensional relation, the subtle issue is that $\barin$ may not literally be set-like, since the $\simeq$-equivalence classes themselves may be proper classes, and in $\ZF^-$ we are not necessarily able to pick representatives globally from the equivalence classes. So it may not be immediately clear that $M$ can undertake the Mostowski collapse. But, we show, it can.

In $M$, let $I$ be the class of pairs $(a,\bar a)$ where $a\in M$ and $\bar a\in\Mbar$ and there is a set $\overbar A\in M$  with (i) $\overbar A\of\Mbar$; (ii) $\overbar A$ is transitive with respect to $\barin$ modulo $\simeq$, in the sense that if $x\mathrel{\barin} y\in\overbar A$, then there is some $x'\in\overbar A$ such that $x\simeq x'$; and (iii) $\<\overbar A,\barin>/\simeq$ is isomorphic to $\<\TC(\singleton{a}),\in>$. Note that if $(a,\bar a)$ are like this, then the isomorphism will be exactly the inverse of the Mostowski collapse of $\<\overbar A,\barin>/\simeq$.

We claim by $\in^{\Mbar}$-induction (externally to $M$) that every $\bar a\in\Mbar$ is associated in this way with some $a\in M$, and furthermore, the association $I$ is simply the map $i:a\mapsto \bar a$. To see this, suppose $\bar a\in\Mbar$ and every $\in^{\Mbar}$-element of $\bar a$ is associated via $I$. Since $i$ is an isomorphism, there is some $a$ such that $i(a)=\bar a$. Inductively, every $\bar b\in^{\Mbar}\bar a$ is associated with some $b$ for which $i(b)=\bar b$, and all such $b$ are necessarily elements of $a$. So $M$ can see that every $\bar b\in^{\Mbar}\bar a$ is associated by the class $I$ with a unique element $b\in a$. For each $b$, there are various sets $\overbar B_b$ with largest element $\bar b$ realizing that $(b,\bar b)\in I$. By the collection axiom, we can find a single set $B$ containing such a set for every $b\in a$, serving as a single witnessing set. Let $\overbar A=\singleton{\bar a}\union B$, where $\bar a$ is a representative of the $\simeq$-class of $i(a)$. This is a set in $M$, it is contained in $\Mbar$ and it is transitive with respect to $\barin$ modulo $\simeq$. Since $\overbar B_b\of\overbar A$, the Mostowski collapse of $\<\overbar A,\barin>/\simeq$ agrees with the collapse of $\<\overbar B_b,\barin>/\simeq$, and so sends $\bar b$ to $b$, for every $b\in a$. So this set $\overbar A$ witnesses that $(a,\bar a)\in I$, and agrees with $i$ on $a$. So the map $i$ is definable in $M$, as desired.
\end{proof}

Notice that theorem \ref{Theorem.Well-founded-model-interpretation-is-definable} is not true if one drops the well-foundedness assumption. For example, on general model-theoretic grounds there must be models $M\satisfies\ZFm$ with nontrivial automorphisms $i:M\to M$, and any such map is an interpretation of $M$ in itself (as itself), but such an automorphism can never be definable. Similarly, if \ZFC\ is consistent, then there is a countable computably saturated model $\<M,\in^M>\satisfies\ZFC$, and such a model contains an isomorphic copy of itself as an element $m$, an observation due to \cite[corollary~3.3]{Schlipf1978:TowardModelTheoryThroughRecursiveSaturation}; see also \cite[lemma~7]{GitmanHamkins2010:NaturalModelOfMultiverseAxioms}. This provides an interpretation of $M$ in itself, using $m$ as a parameter, but the isomorphism cannot be definable nor even amenable to $M$, since $M$ cannot allow a definable injection from the universe to a set.

\begin{corollary}\label{Corollary.Mutual-implies-bi-for-models-of-ZF^-}
Every instance of mutual interpretation amongst well-founded models of $\ZF^-$ is a bi-interpretation. Indeed, if $M$ is a well-founded model of $\ZF^-$ and mutually interpreted with any structure $N$ of any theory, as in the figure below, then the isomorphism $i:M\to\Mbar$ is definable in $M$.
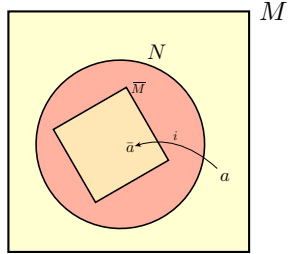
$$
\begin{tikzpicture}[scale=.8,thick]
\drawM{$M$}
\node (Ma) at (2,-1) {};
\node (Mc) at (2,-.5) {};
\node (Mp) at (1.6,-.75) {};
\begin{scope}[scale=.7,shift={(-.2,-.3)},rotate=30,semithick,every node/.style={scale=.8}]
\drawN{$N$}
\node (Nb) at (2,0) {};
\begin{scope}[scale=.5,shift={(-.4,.2)},every node/.style={scale=.5},semithick=false]
\drawM{$\Mbar$}
\node (Md) at (2,-1) {};
\node (Mq) at (.5,-.5) {};
\end{scope}
\end{scope}
\draw[-{[scale=.5]>},>=Stealth,bend right,very thin,shorten >=2pt,shorten <=-1pt] (Mp) node[scale=.7] {$a$} to node[midway, above,scale=.5] {$i$} (Mq)  node[scale=.6] {$\bar a$};
\end{tikzpicture}$$
\end{corollary}

\begin{proof}
This follows immediately from theorem \ref{Theorem.Well-founded-model-interpretation-is-definable}, since any instance of mutual interpretation leads to an instance of a model being interpreted inside itself, and so by theorem \ref{Theorem.Well-founded-model-interpretation-is-definable} the interpretation maps will be definable.
\end{proof}

We are unsure whether we can achieve the half-way situation, where a transitive model $M\satisfies\ZF^-$ is mutually interpreted with some nonstandard model $N\satisfies\ZF^-$, without being bi-interpreted.

\section{Bi-interpretation does not occur in models of \ZF}\label{Section.Bi-interpretation-does-not-occur}

Finally, we come to the heart of the paper, concerning the extent of bi-interpretation in set theory. Let us begin with Enayat's theorem, which will show that distinct models of \ZF\ are never bi-interpretable. As we mentioned earlier, a theory $T$ is \emph{semantically tight}, if any two bi-interpretable models of $T$ are isomorphic. Strengthening this, a theory $T$ is \emph{solid}, if whenever $M$ and $N$ are mutually interpreted models of $T$ and there is an $M$-definable isomorphism of $M$ with its copy $\Mbar$ inside the copy of $N$ defined in $M$, then $M$ and $N$ are isomorphic. Thus, solidity strengthens semantic tightness, because it requires the models to be isomorphic even when one has essentially only half of a bi-interpretation, in that only one side of the interpretation needs to be definable in the model, rather than both.

\begin{theorem}[Enayat \cite{enayat2017variations}]\label{ali-theorem}
\ZF\ is solid. 
\end{theorem}

\begin{proof}
Assume that $M$ and $N$ are mutually interpreted models of \ZF, where $N$ is definable in $M$ and $\Mbar$ is a definable copy of $M$ inside $N$, with an isomorphism $i:M\to\Mbar$ that is definable in $M$, as illustrated here:

$$
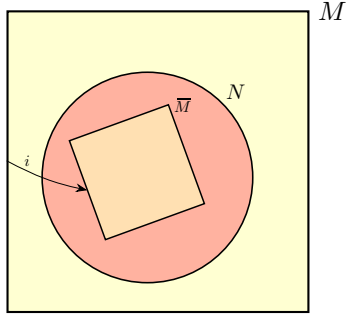
\begin{tikzpicture}[scale=1,thick]
\begin{scope}[thick,shift={(7.5,0)}]
\drawM{$M$}
\begin{scope}[scale=.7,shift={(-.2,-.3)},every node/.style={scale=.8},semithick]
\drawN{$N$}
\begin{scope}[scale=.5,shift={(-.4,.2)},rotate=20,every node/.style={scale=.6},line width=.2mm]
\drawMbi{$\Mbar$}{$i$}{1}
\begin{scope}[scale=.7,shift={(-.2,-.3)},every node/.style={scale=.5},thin]
\end{scope}
\end{scope}
\end{scope}
\end{scope}
\end{tikzpicture}$$

Suppose that $N$ thinks $A$ is a nonempty subclass of $\Mbar$. Consider the pre-image $A'=\set{i^{-1}(x)\mid x\in^N A}$, which is a nonempty class in $M$. So it must have an $\in^M$-minimal element $a \in^M A'$. It follows that $i(a)$ is an $\in^\Mbar$-minimal element of $A$, and so $N$ sees $\in^\Mbar$ as well founded. Because of this fact, the ordinals of the models will be comparable, and they will have some ordinals isomorphically in common. We claim that for these common ordinals, the rank-initial segments of $M$, $N$ and $\Mbar$ are isomorphic. Specifically, if $\alpha$ is an ordinal in $M$, which is isomorphic to the ordinal $\overbar\alpha=i(\alpha)$ in $\Mbar$, and this happens to be isomorphic in $N$ to an ordinal $\alpha^*$ in $N$, then we claim that $N$ can see that $\<V_{\alpha^*},\in>^N$ is isomorphic to $\<V_{\overbar\alpha},\in>^\Mbar$, which we know is isomorphic to $\<V_\alpha,\in>^M$. This is certainly true when $\alpha=0$, and if it is true at $\alpha$, then it will also be true at $\alpha+1$, because every subset of $V_\alpha^N$ exists also as a class in $M$ and therefore can be pushed via $i$ to the corresponding subset of $V_{\overbar\alpha}^\Mbar$; and conversely, every subset of $V_{\overbar\alpha}^\Mbar$ in $\Mbar$ exists also as a class in $N$ and can be pulled back by the isomorphism with $V_{\alpha^*}^N$. So in $N$ we may extend the isomorphism canonically from $V_{\alpha^*}^N\cong V_{\overbar\alpha}^\Mbar$ to $V_{\alpha^*+1}^N\cong V_{\overbar\alpha+1}^\Mbar$. Furthermore, because $V_{\alpha^*}^N$ is transitive in $N$ and transitive sets are rigid, the isomorphisms must be unique, and so at limit stages the isomorphism will simply be the coherent union of the isomorphisms at the earlier stages.

If the ordinals of $N$ run out before the ordinals of $M$, then $\Ord^N$ is isomorphic to an ordinal $\lambda$ in $M$ and hence to $i(\lambda)=\overbar\lambda$ in $\Mbar$. In this case, the isomorphism from the previous paragraph will mean that $N=V_{\Ord}^N\cong V_{\overbar\lambda}^\Mbar$, which is isomorphic to $V_\lambda^M$. So $M$ will see that $N$ is bijective with a set. But $M$ has a bijection of itself with $\Mbar$, which is contained in $N$. So $M$ will think that the universe is bijective with a set, which is impossible in \ZF. 

If in contrast the ordinals of $M$ and hence $\Mbar$ run out before the ordinals of $N$, then $N$ would recognize that $\Mbar$ is isomorphic to a transitive set $V^N_{\alpha^*}$. Since it is a set, $N$ has a truth predicate for it, and so $N$ can define a truth predicate for $\Mbar$. Since $N$ and $i$ are definable in $M$, we can pull this truth predicate back via $i$ to produce an $M$-definable truth predicate on $M$, contrary to Tarski's theorem on the nondefinability of truth. 

So the ordinals of the three models must be exactly isomorphic, and so the isomorphisms of the rank-initial segments of the models in fact produce an isomorphism of $N=V_{\Ord}^N$ with $\Mbar=V_{\Ord}^\Mbar$ and hence with $M$, as desired.
\end{proof}

Albert Visser had initially proved a corresponding result in 2004 for Peano arithmetic \PA\ in \cite[see pp.52--55]{visser2004categories},\cite{Visser2006:Categories-of-theories-and-interpretations}. Ali Enayat had followed with a nice model-theoretic argument showing specifically that \ZF\ and \ZFC\ are not bi-interpretable, using the fact that \ZFC\ models can have no involutions in their automorphism groups, but \ZF\ models can, before finally proving the general version of his theorem, for \ZF, for second-order arithmetic $Z_2$ and for second-order set theory \KM\ in \cite{enayat2017variations}. The ZF version was apparently also observed independently by Harvey Friedman and Albert Visser, by Fedor Pakhomov, as well as by ourselves \cite{Hamkins.blog2018:Different-set-theories-are-never-biinterpretable}, before we had realized that this was a rediscovery.

Enayat defines that a theory $T$ is \emph{tight}, if whenever two extensions of $T$ are bi-interpretable, then they are the same theory.

\begin{corollary}\label{Corollary.ZF-is-tight}
$\ZF$ is tight. That is, no two distinct set theories extending \ZF\ are bi-interpretable.
\end{corollary}

This corollary follows from theorem \ref{ali-theorem}, because in fact every solid theory is tight. If $T$ is solid, and $T_1$, $T_2$ are distinct bi-interpretable extensions of $T$, then consider any model $M\satisfies T_1$. By the bi-interpretability of the theories, $M$ is bi-interpretable with a model $N\satisfies T_2$. Since these are both models of $T$, it follows by the solidity of $T$ that $M$ and $N$ are isomorphic, and so $M$ is a model of $T_2$. Since the argument also works conversely, the two theories are the same.

In particular, \ZF\ is not bi-interpretable with \ZFC, nor with \ZFC+\CH, nor $\ZFC+\neg\CH$ and so on. Theorem \ref{Theorem.Set-theories-are-mutually-interpretable} therefore cannot be strengthened from mutual interpretation to bi-interpretation, and there is no nontrivial bi-interpretation phenomenon in set theory amongst the models or theories strengthening \ZF.

Furthermore, we claim that there is no mutual interpretation phenomenon amongst the well-founded models of \ZF.

\begin{theorem}\label{Theorem.Well-founded-models-are-never-mutually-interpretable}
Nonisomorphic well-founded models of \ZF\ are never mutually interpretable.
\end{theorem}

\begin{proof}
Corollary \ref{Corollary.Mutual-implies-bi-for-models-of-ZF^-} shows that every instance of mutual interpretation amongst the well-founded models of $\ZF$ is a bi-interpretation, but theorem \ref{ali-theorem} shows that bi-interpretation amongst models of \ZF\ occurs only between isomorphic models.
\end{proof}

Let us now explain how to deduce the solidity of \ZF\ by means of the following ``internal categoricity'' theorem of Jouko \Vaananen. He had stated the result in \cite{vaananen2019extension} for \ZFC, but his argument did not use \AC, and so we state and prove it here for \ZF.

\begin{theorem}[V\"a\"an\"anen \cite{vaananen2019extension}]\label{Theorem.Vaananen}
Assume that $\<V,\in,\barin>$ is a model of \ZF\ with respect to both membership relations $\in$ and $\barin$, in the common language. More precisely, it is a model of $\ZF_\in(\barin)$, using $\in$ as the membership relation and $\barin$ as a class predicate and also of $\ZF_{\barin}(\in)$, using $\barin$ as the membership relation and $\in$ as a class predicate. Then $\<V,\in>\iso\<V,\barin>$, and furthermore, there is a unique definable isomorphism in $\<V,\in,\barin>$.
\end{theorem}

\begin{proof}
Let us begin by observing that by the foundation axiom, every nonempty definable $\in$-class will admit an $\in$-minimal element, and similarly every nonempty definable $\barin$-class will have an $\barin$-minimal element. Thus, both $\in$ and $\barin$ will be well-founded relations in this model, regardless of whether we use $\in$ or $\barin$ as the membership relation. In particular, the orders $\langle\Ord^{\<V,\in>},\in\rangle$ and $\langle\Ord^{\<V,\barin>},\barin\rangle$ will be comparable well-ordered classes, and we may assume without loss that the former is isomorphic to an initial segment of the latter. Thus, for every $\in$-ordinal $\alpha$, there is an $\barin$-ordinal $\overbar\alpha$ to which it is isomorphic, and furthermore this ordinal is unique, and the isomorphism is unique. We claim that $\<V_\alpha,\in>\iso\<\overbar V_{\overbar\alpha},\barin>$, by induction on $\alpha$. Since these are transitive sets and hence rigid, the isomorphism is unique. The claim is clearly true for $\alpha=0$. It remains true through limit ordinals, since the isomorphism in that case is simply the union of the previous isomorphisms. At successor stages, we need only observe that any isomorphism of $\<V_\alpha,\in>$ with $\<\overbar V_{\overbar\alpha},\barin>$ extends to an isomorphism of the power set, since any $\in$-subset of $V_\alpha$ transfers to an $\barin$-subset of $\overbar V_{\overbar\alpha}$ by pointwise application of the isomorphism, and vice versa. If $\Ord^{\<V,\barin>}$ is taller than $\Ord^{\<V,\in>}$, then we can see that $\<V,\in>\iso\<\overbar V_{\gamma},\barin>$, where $\gamma\in\Ord^{\<\overbar V,\barin>}$ is the first $\barin$-ordinal not arising from an $\in$-ordinal. But now $\overbar V_{\gamma+1}$, which is the $\barin$-power set of $\overbar V_\gamma$, is a subclass of $V$ and therefore injects into $\overbar V_\gamma$, contrary to Cantor's theorem. So the ordinals are the same height, and thus we have achieved $\<V,\in>\iso\<V,\barin>$. This isomorphism is definable, since every $V_\alpha$ is isomorphic to $\overbar V_{\overbar\alpha}$ by a unique isomorphism.
\end{proof}

Theorem \ref{ali-theorem}, statements (1) and (2), can be seen as a consequence of this theorem---and indeed the proofs we gave are roughly analogous---because if two models of \ZF\ are bi-interpretable, then by the results mentioned in section \ref{Section.Background-on-interpretating-in-set-theory}, we do not need tuples or nonidentity equivalence relations. So the interpretations will interpret each model as a subclass of the other. By the class version of the Schr\"oder-Cantor-Bernstein theorem, these injections can be transformed into bijection of the models, and in this way the bi-interpretation can be transformed into a bi-interpretation synonymy (see also \cite{friedman2014bi}). Thus, we reduce to the case of two membership relations $\<M,\in,\barin>$ on the whole universe, with each relation being definable when using the other as the membership relation. Because of this, the model will satisfy the theory $\ZF(\in,\barin)$, and consequently the models will be isomorphic $\<M,\in>\cong\<M,\barin>$ by \Vaananen's theorem.

\section{\Vaananen\ internal categoricity fails for $\ZFCm$}

Theorems \ref{ali-theorem} and \ref{Theorem.Vaananen} were both proved with arguments that made a fundamental use of the $V_\alpha$ hierarchy, thereby relying on the power set axiom. Was the use of the $V_\alpha$ hierarchy significant? Can one prove the analogues of the theorems for $\ZFCm$, that is, for set theory without the power set axiom? (See \cite{GitmanHamkinsJohnstone2016:WhatIsTheTheoryZFC-Powerset?} for a subtlety about exactly what the theory $\ZFCm$ is; one should include the collection axiom, and not just replacement.) The answer is negative. We shall prove that neither theorem is true for $\ZFCm$, not even in the case of well-founded models. Let us begin by proving that the \Vaananen\ internal categoricity theorem fails for $\ZFCm$.

\begin{theorem}\label{Theorem.ZFCm-double-relation-model}
There is a transitive set $M$ and an alternative well-founded membership relation $\barin$ on $M$, such that $\<M,\in,\barin>$ satisfies $\ZFCm(\in,\barin)$, that is, $\ZFCm$ using either $\in$ or $\barin$ as the membership relation and allowing both as class predicates, such that $\<M,\in>$ is not isomorphic to $\<M,\barin>$, although both are well-founded.
\end{theorem}

\begin{proof}
Let us argue first merely that the situation is consistent with \ZFC. Assume that Luzin's hypothesis holds, that is, $2^\omega=2^{\omega_1}$; for example, this holds after adding $\omega_2$ many Cohen reals over a model of \GCH, since in this case we would have $2^\omega=2^{\omega_1}=\omega_2$. It follows that $H_{\omega_1}$ and $H_{\omega_2}$ are equinumerous, and so there is a bijection $\pi:H_{\omega_1}\to H_{\omega_2}$. Let $\tilde\in$ be the image of $\in\restrict H_{\omega_1}$ under this bijection, and let $\barin$ be the preimage of $\in\restrict H_{\omega_2}$. Thus, $\pi$ is an isomorphism of the structures
 $$\pi:\<H_{\omega_1},\in,\barin>\cong \<H_{\omega_2},\tilde\in,\in>.$$
The first of these is a model of $\ZFCm(\barin)$, using $\in$ as membership, since $H_{\omega_1}$ is a model of $\ZFCm$ with respect to any predicate. Similarly, the second is a model of $\ZFCm(\tilde\in)$, again using $\in$ as membership, since we can likewise augment $H_{\omega_2}$ with any predicate. By following the isomorphism, an equivalent way to say this is that $\<H_{\omega_1},\in,\barin>$ is a model of $\ZFCm_\in(\barin)$, using $\in$ as membership and $\barin$ as predicate, and also a model of $\ZFCm_{\barin}(\in)$, using $\barin$ as membership this time and $\in$ as a class predicate. But $\<H_{\omega_1},\in>$ is not isomorphic to $\<H_{\omega_1},\barin>$, because the first thinks every set is countable and the second does not. So this model is just as desired.

To get outright existence of the models, observe first that by taking an elementary substructure, we can find a countable model of the desired form in the forcing extension. Furthermore, the existence of a countable well-founded model $\<M,\in,\barin>$ with the desired properties is a $\Sigma^1_2$ assertion, which is therefore absolute to the forcing extension by Shoenfield’s absoluteness theorem. And so therefore there must have already been such an example in the ground model, without need for any forcing.\end{proof}

\section{$\ZFCm$ is neither solid nor tight}

In order to show that $\ZFCm$ is not solid, we need to provide different models of $\ZFCm$ that are bi-interpretable, but not isomorphic. Let us begin by describing merely how this can happen in some models of set theory. For this, we shall employ a more refined version of the argument used to prove theorem \ref{Theorem.ZFCm-double-relation-model}.

\begin{theorem}\label{Theorem.Homega1-Homega_2-bi-interpretable}
In the Solovay-Tennenbaum model of $\MA+\neg\CH$ obtained by c.c.c.~forcing over the constructible universe, the structures $\<H_{\omega_1},\in>$ and $\<H_{\omega_2},\in>$ are bi-interpretable. Thus, there can be two well-founded models of $\ZFCm$ that are bi-interpretable, but not isomorphic.
\end{theorem}

The features we require in the Solovay-Tennenbaum model are the following:
\begin{itemize}
  \item $H_{\omega_1}$ has a definable almost disjoint $\omega_1$-sequence of reals;
  \item every subset $A\of\omega_1$ is coded by a real via almost-disjoint coding with respect to this sequence.
\end{itemize}
In other words, there should be a definable $\omega_1$-sequence $\<a_\alpha\mid\alpha<\omega_1>$ of infinite subsets $a_\alpha\of\omega$ with any two having finite intersection; and for every $A\of\omega_1$ there should be $a\of\omega$ for which $\alpha\in A\iff\ a\intersect a_\alpha$ is infinite. These properties are true in the Solovay-Tennenbaum model $L[G]$, because we can define the almost disjoint family in $L$, and by Martin's axiom every subset of $\omega_1$ is coded with respect to it by almost-disjoint coding.

\begin{proof}
Let us assume we have the two properties isolated above. The structure $\<H_{\omega_1},\in>$, of course, is a definable substructure of $\<H_{\omega_2},\in>$, which gives one direction of the interpretation. It remains to interpret $\<H_{\omega_2},\in>$ inside $\<H_{\omega_1},\in>$ and to prove that this is a bi-interpretation. The second assumption above implies, of course, that $2^\omega=2^{\omega_1}$, and so at least the two structures have the same cardinality. But more, because the almost-disjoint sequence is definable in $H_{\omega_1}$ and all the reals are there, the structure $\<H_{\omega_1},\in>$ in effect has access to the full power set $P(\omega_1)$; the subsets $A\of\omega_1$ are uniformly definable in the structure $\<H_{\omega_1},\in>$ from real parameters. And furthermore, every set in $H_{\omega_2}$ is coded by a subset of $\omega_1$. Specifically, if $x\in H_{\omega_2}$ then $x$ is an element of a transitive set $t$ of size $\omega_1$, with $\<t,\in>\cong\<\omega_1,E>$ for some well-founded extensional relation $E$ on $\omega_1$. The set $x$ can be in effect coded by specifying $E$ and the ordinal $\alpha$ representing $x$ in the structure $\<\omega_1,E>$. Since $H_{\omega_1}$ contains all countable sequences of ordinals, it can correctly identify which relations $E$ on $\omega_1$, as encoded by a real via the almost-disjoint coding, are well-founded and extensional. Let us define an equivalence relation
$a\simeq a'$ for reals $a,a'\of\omega$ coding such pairs $(\alpha,E)$ and $(\alpha',E')$ that represent the same set, which happens just in case there is an isomorphism of the hereditary $E$-predecessors of $\alpha$ to the hereditary $E'$-predecessors of $\alpha'$; in other words, when the transitive closures of the corresponding encoded sets $\singleton{x}$ and $\singleton{x'}$ are isomorphic. The isomorphism itself is a map of size at most $\omega_1$, which can therefore be seen definably in the structure $\<H_{\omega_1},\in>$. And we can recognize when the set $x$ coded by $(\alpha,E)$ is an element of the set $y$ coded by $(\beta,F)$, since this just means that $(\alpha,E)$ is equivalent to $(\gamma,F)$ for some $\gamma\mathrel{F}\beta$. Since every element of $H_{\omega_2}$ is coded in this way by relations on $\omega_1$ and hence ultimately by reals $a\of\omega$, we see that $\<H_{\omega_2},\in>$ is interpreted in $\<H_{\omega_1},\in>$ by the codes $(\alpha,E)$ for well-founded extensional relations $E$ on $\omega_1$, which are themselves coded by reals via almost-disjoint coding. Ultimately, objects $x\in H_{\omega_2}$ are represented by a real $a\of$ coding a pair $(\alpha,E)$ which codes a transitive set having $x$ as an element at index $\alpha$ in that coding.

This interpretation is a bi-interpretation, because $H_{\omega_1}$ can construct suitable codes for any object $x\in H_{\omega_1}$, and thereby recognize how itself arises within the interpreted copy of $H_{\omega_2}$. And conversely, $H_{\omega_2}$ can define $H_{\omega_1}$ as a submodel and it can see how the objects in $H_{\omega_2}$ are represented inside that model via almost disjoint coding. The models $\<H_{\omega_1},\in>$ and $\<H_{\omega_2},\in>$ are therefore well-founded models of $\ZFCm$ that are bi-interpretable, but they are not isomorphic, because the first model thinks every set is countable and the second does not.
\end{proof}

We can improve the previous theorem to achieve an actual synonymy of the two structures as follows.

\begin{theorem}\label{Theorem.Homega1-Homega2-synonymy}
If \ZFC\ is consistent, then it is consistent that there is a membership relation $\barin$ definable in $\<H_{\omega_1},\in>$ such that $\<H_{\omega_1},\barin>\cong\<H_{\omega_2},\in>$, putting these structures into bi-interpretation synonymy.
\end{theorem}

\begin{proof}
The argument relies on a result of Leo Harrington \cite{Harrington1977:Long-projective-wellorderings}, showing that it is relatively consistent with \ZFC\ that $\MA+\neg\CH$ hold, yet there is a projectively definable well-ordering of the reals; in fact, Harrington achieves this in a forcing extension of $L$. (Many thanks to Gabe Goldberg \cite{Goldberg2020.MO350585:Can-Homega1-Homega2-be-in-synonymy?} for pointing out Harrington's paper in response to our MathOverflow question concerning the possibility of this theorem.) In such a model, we can choose least elements within each equivalence class of reals in the interpretation used in the proof of theorem \ref{Theorem.Homega1-Homega_2-bi-interpretable}, and thereby omit the need for an equivalence relation in the first place. So every element of $H_{\omega_2}$ will be coded by a real. Since the decoding of hereditarily countable sets can be undertaken inside $H_{\omega_1}$, we thereby have a definable injection of $H_{\omega_1}$ into the coding reals, and so by the Cantor-Schröder–Bernstein theorem, $H_{\omega_1}$ is bijective with the coding reals. Because the Cantor-Schröder–Bernstein theorem is sufficiently constructive, this bijection is definable inside $\<H_{\omega_1},\in>$, and so we may pull back the interpreted membership relation on the coding reals to get a definable relation $\barin$ on all of $H_{\omega_1}$, such that $\<H_{\omega_1},\barin>\cong\<H_{\omega_2},\in>$, as desired. This is a bi-interpretation synonymy, since $H_{\omega_1}$ can see how it is represented in that coding, and $H_{\omega_2}$ can define $H_{\omega_1}$ and also see how its elements are coded.
\end{proof}

We find it interesting to compare theorems \ref{Theorem.Homega1-Homega_2-bi-interpretable} and \ref{Theorem.Homega1-Homega2-synonymy} with \cite[corollary~2.5.1]{enayat2017variations}, which says that $\ZFCm+$``every set is countable'' is solid, and similarly corollary 2.7.1, which asserts that $\ZFCm+``V=H_{\kappa^+}$ for some inaccessible cardinal $\kappa$'' also is solid.

The following theorem explains somewhat the need in theorem \ref{Theorem.Homega1-Homega_2-bi-interpretable} for assuming that we were working close to $L$.

\begin{theorem}
If there is no projectively definable $\omega_1$-sequence of distinct reals, then $\<H_{\omega_2},\in>$ cannot be interpreted in $\<H_{\omega_1},\in>$. In particular, in this case the structures are not bi-interpretable nor even mutually interpretable.
\end{theorem}

The hypothesis is a direct consequence of sufficient large cardinals, since it is a consequence of $\AD^{L(\R)}$. But meanwhile, it can also be forced by the \Levy\ collapse of an inaccessible cardinal; and since it implies $\omega_1$ is inaccessible to reals, it is equiconsistent with an inaccessible cardinal.

\begin{proof}
If $H_{\omega_2}$ were interpreted in $H_{\omega_1}$, then since the former structure definitely has an $\omega_1$-sequence of distinct reals, such a sequence would be definable (from parameters) in the structure $\<H_{\omega_1},\in>$. But this latter structure is interpreted in the reals in second-order arithmetic, by coding hereditarily countable sets by relations on $\omega$. By means of this interpretation, first-order assertions in $\<H_{\omega_1},{\in}>$ can be viewed as projective assertions. And so there would be a projectively definable $\omega_1$-sequence of distinct reals, contrary to our assumption.
\end{proof}

Meanwhile, we can extend theorem \ref{Theorem.Homega1-Homega_2-bi-interpretable} from the consistency result to prove outright in \ZFC\ that there are always transitive counterexample models to be found.

\begin{theorem}\label{Theorem.ZFCm-is-not-solid}
 The theory $\ZFCm$ is not solid, and not even semantically tight, not even for well-founded models. Indeed, there are transitive models $\<M,\in>$, $\<N,\in>$ of $\ZFCm$ that form a bi-interpretation synonymy, but are not isomorphic.
\end{theorem}

\begin{proof}
By theorem \ref{Theorem.Homega1-Homega2-synonymy}, there are such transitive sets in a forcing extension of $L$. In that model, by taking a countable elementary substructure, we can find countable transitive sets with the desired feature. And furthermore, the assertion that there are such countable transitive sets like that, for which the particular bi-interpretation definitions work and form a synonymy, is a statement of complexity $\Sigma^1_2$. By Shoenfield absoluteness, this statement must already be true in $L$, and hence also in $V$. So there are countable transitive models $\<M,\in>$ and $\<N,\in>$ of $\ZFCm$ that form a bi-interpretation synonymy, but are not isomorphic.
\end{proof}

And we can also use the argument to show that $\ZFCm$ is not tight.

\begin{theorem}\label{Theorem.ZFCm-is-not-tight}
$\ZFCm$ is not tight.
\end{theorem}

\begin{proof}
We shall find two extensions of $\ZFCm$ that are bi-interpretable, but not the same theory. Let us take $T_1$ and $T_2$ be the theories describing the situation of $\<H_{\omega_1},\in>$ and $\<H_{\omega_2},\in>$ in theorem \ref{Theorem.Homega1-Homega_2-bi-interpretable}. Specifically, let $T_2$ assert $\ZFCm$ plus the assertion that $\omega_1$ exists but not $\omega_2$, that $\omega_1=\omega_1^L$, that $\omega_2=\omega_2^L$, and that every subset of $\omega_1$ is coded by a real using the almost-disjoint coding with respect to the $L$-least almost-disjoint family $\<a_\alpha\mid\alpha<\omega_1>$. Let $T_1$ be the theory $\ZFCm$ plus the assertion that every set is countable and that the interpretation of $H_{\omega_2}$ in $H_{\omega_1}$ used in the proof of theorem \ref{Theorem.Homega1-Homega_2-bi-interpretable} defines a model of $T_2$. That is, $T_1$ asserts that the interpretation does indeed interpret a model of $T_2$. These two theories are bi-interpretable, using the interpretations provided in the proof of theorem \ref{Theorem.Homega1-Homega_2-bi-interpretable}, because any model of $T_1$ interprets a model of $T_2$, and any model of $T_2$ can define its $H_{\omega_1}$, which will be a model of $T_1$, and the composition maps are definable just as before. \end{proof}

\section{Zermelo set theory is neither solid nor tight}

We should like now to consider the case of Zermelo set theory. We shall prove that nontrivial instances of bi-interpretation occur in models of Zermelo set theory, and consequently Zermelo set theory is neither solid nor tight. Specifically, we shall prove the following:

\begin{theorem}\label{Theorem.Zermelo-not-solid-not-tight}\
 \begin{enumerate}
  \item {\rm Z} is not semantically tight (and hence not solid), not even for well-founded models: there are bi-interpretable well-founded models of Zermelo set theory that are not isomorphic.
  \item Every model of \ZF\ is bi-interpretable with a transitive inner model of Zermelo set theory, in which the replacement axiom fails.
  \item {\rm Z} is not tight: there are distinct bi-interpretable strengthenings of Z.
 \end{enumerate}
\end{theorem}

Our argument will make fundamental use of the model-construction method of Adrian Mathias \cite{mathias2001slim}. Mathias had used his methods to construct diverse interesting models of Zermelo set theory, some of them quite peculiar, such as a transitive inner model of Zermelo set theory containing all the ordinals, but not having $V_\omega$ as an element. Let us briefly review his method. The central definition is that a class $C$ is \emph{fruitful}, if
\begin{enumerate}
  \item every $x\in C$ is transitive; 
  \item $\Ord\subseteq C$;
  \item $x \in C$ and $y \in C$ implies $x \cup y \in C$;
  \item $x \in C$ and $y \subseteq P(x)$ implies $x \cup y \in C$.
\end{enumerate}
Fruitful classes, Mathias proves, lead to transitive models of Zermelo set theory as follows.

\begin{theorem}[Mathias {\cite[prop.~1.2]{mathias2001slim}}]\label{Theorem.Mathias-fruitful}
  If $C$ is fruitful, then $M=\Union C$ is a supertransitive model of Zermelo set theory with the foundation axiom.
\end{theorem}

A transitive set $M$ is \emph{supertransitive}, if it contains every subset of its elements, that is, if $x\of y\in M\implies x\in M$. In the context of Zermelo set theory, we take the foundation axiom in the form of the $\in$-induction scheme, which is equivalent over Z to the ``transitive containment'' assertion that every set is an element of a transitive set.

The importance of the theorem is that Mathias explains how to construct a great variety of fruitful classes $T^{Q,G}$, defined in terms of a  function $Q:\omega\to V_\omega$ and a family $G$ of functions from $\omega$ to $\omega$, which specify rates of growth for the allowed sets.  Namely, a transitive $x$ is in $T^{Q,G}$ just in case its rate-of-growth function
 $$f_x^Q(n)=|x\intersect Q(n)|$$
is in $G$. Under general hypotheses, he proves, $T^{Q,G}$ is fruitful.

For our application, it will suffice to consider one of his key examples, described in \cite[\S3]{mathias2001slim}. Namely, we consider the case of $T^{R,F}$, where $R(n)=V_n$ is the rank-initial segment function and $F$ is the collection of functions $f$ bounded by one of the superexponential functions $b_k$ of the form
\def\rddots#1{\cdot^{\cdot^{\cdot^{#1}}}}
$$b_k:n\mapsto 2^{2^{\rddots {2^n}}}\raise3pt\hbox{$\Bigr\}{\tiny k}$}$$
for some $k\in\omega$. Mathias proves in this instance that $T^{R,F}$ is fruitful, and therefore by theorem \ref{Theorem.Mathias-fruitful} the class $M=\Union T^{R,F}$ is a supertransitive model of Zermelo set theory. The elements of $M$ are precisely those sets $x$ whose transitive closure has a rate of growth in $V_n$ bounded by some $b_k$.
 $$x\in M\qquad\Iff\qquad \exists k\ \forall n\quad |\TC(\singleton{x})\intersect V_n|\ \leq\  b_k(n).$$
Since each $b_k$ has a fixed superexponential depth $k$, it follows that each of them has strictly slower asymptotic growth than the full tetration function
 $$n\mapsto |V_n|=2^{2^{\rddots 2}}\raise3pt\hbox{$\Bigr\}{\tiny n-1}$},$$
and therefore the set $V_\omega$ is not in $M$ because it grows too quickly, \cite[thm.~3.8]{mathias2001slim}.

In order to prove the claims of theorem \ref{Theorem.Zermelo-not-solid-not-tight}, we shall show that the original \ZF\ model $\<V,\in>$ is bi-interpretable with the Zermelo model $\<M,\in>$ defined by $M=\Union T^{R,F}$ above. We already have one direction of the interpretation, since we have defined $M$ as a transitive inner model inside $V$. The difficult part will be conversely to provide an interpretation of $\<V,\in>$ inside $\<M,\in>$. For this, we shall make use of the \emph{Zermelo tower} hierarchy, defined for any set $a$ in $V$ as follows:
\begin{enumerate}
  \item $V^{(a)}_0 = \emptyset$;
  \item $V^{(a)}_{\alpha + 1} =
  \{a\} \cup \bigl(P(V^{(a)}_{\alpha}) - \{\emptyset\}\bigr)$, for successor ordinals;
  \item $V^{(a)}_{\lambda} = \bigcup\limits_{\alpha<\lambda} V^{(a)}_\alpha$, for limit ordinals $\lambda$.
\end{enumerate}
Mathias had introduced the Zermelo tower idea in \cite[\S4]{mathias2001slim}, but defined and considered it only at finite stages and $\omega$, whereas we continue the construction, crucially, through all the ordinals. Our reason for doing so is that the full Zermelo tower $V^{(a)}$ is a definable copy of the original universe, as we prove here:

\begin{lemma}
 For any set $a$, the universe $\<V,\in>$ is definably isomorphic to the full Zermelo tower $\<V^{(a)},\in>$.
\end{lemma}

\begin{proof}
The isomorphism from $V$ to $V^{(a)}$ is simply the operation replacing every hereditary instance of $\emptyset$ in a set with $a$. Specifically, we define the replacing operation $x\mapsto x^{(a)}$ as follows:
\begin{align*}
\emptyset^{(a)} &= a.\\
x^{(a)} &= \{y^{(a)} \mid y\in x\}
\end{align*}
A simple argument by transfinite induction now shows that this is an isomorphism of every $V_\alpha$ with $V^{(a)}_\alpha$, which establishes the theorem.
\end{proof}

Next, we show that for suitable sets $a$, the Zermelo tower is contained within the Zermelo universe $M$ we defined above.

\begin{lemma}\label{Lemma.V^(a)-subset-M}
If $a$ is any infinite transitive set in $M$, then $V^{(a)}\of M$.
\end{lemma}

\begin{proof}
Suppose that $a$ is an infinite transitive set in $M$. It therefore obeys the growth rate requirement, and so there is some $k$ for which $|a\intersect V_n|\leq b_k(n)$ for every $n$. Notice that $a\union V_\alpha^{(a)}$ is a transitive set, and since (i) $u^{(a)}$ is infinite for every set $u$, and (ii) every element of $V_n$ is hereditarily finite, it follows that
$$(a\union V_\alpha^{(a)})\intersect V_n=a\intersect V_n.$$
Consequently, $a\union V_\alpha^{(a)}$ obeys the growth-rate requirement, and consequently every $V^{(a)}_\alpha$ is in $M$. So $V^{(a)}\of M$.
\end{proof}

\begin{lemma}
For any set $a\in M$, the class $V^{(a)}$ is definable in $M$ from parameter~$a$.
\end{lemma}

\begin{proof}
One might naturally want to define $V^{(a)}$ in $M$ by transfinite recursion, as we did in defining the hierarchy $V^{(a)}_\alpha$ above. The problem with this is that $M$ is a model merely of Zermelo set theory, in which such definitions by transfinite recursion do not necessarily succeed. Nevertheless, in instances as here where we know independently that the recursion does have a solution in $M$, the recursive definition does in fact succeed in defining it.

But instead of that definition, let us consider another simple definition. Define that a \emph{terminal $\in$-descent} from a set $x$ is a finite $\in$-descending sequence of sets from $x$ to $\emptyset$
 $$x=x_n\ni x_{n-1}\ni\cdots\ni x_1\ni x_0=\emptyset.$$
We claim that $V^{(a)}$ consists in $M$ precisely of those sets $x$ for which every terminal $\in$-descent passes through the set $a$, in the sense that $a=x_i$ for some $i$. Certainly every $x\in V^{(a)}$ is like this, since if $y\in x\in V^{(a)}_{\alpha+1}$, then either $x=a$ or $y\in V^{(a)}_\alpha$, and by induction the terminal descents from $y$ will pass through $a$. Conversely, assume that every terminal $\in$-descent from $x$ passes through $a$. So either $x=a$, in which case it is in $V^{(a)}$, or else the assumption about terminal descents is also true of the elements of $x$. By $\in$-induction, therefore, the elements of $x$ are all in $V^{(a)}$, and this implies that $x$ is in $V^{(a)}_\gamma$ at an ordinal stage $\gamma$ beyond the stages where the elements of $x$ are placed into $V^{(a)}$.
\end{proof}

If we take $a=\omega$, or some other definable infinite transitive set in $M$, then we don't need $a$ as a parameter, and so we've proved that $\<V,\in>$ and $\<M,\in>$ are at least mutually interpretable; each has a definable copy in the other.

\begin{lemma}
 The mutual interpretation of $\<V,\in>$ and $\<M,\in>$ is a bi-interpretation.
\end{lemma}

\begin{proof}
We've seen already that $M$ is a definable transitive inner model of $V$, and $\<V,\in>$ is definably isomorphic to its copy $\<V^{(a)},\in>$ in $M$, isomorphic by the map $x\mapsto x^{(a)}$, which is definable in $V$, assuming as we have that we use $a=\omega$ or some other definable infinite transitive set in $M$.

Conversely, we've seen that $M$ can define the model $\<V^{(a)},\in>$. Inside that model, we define the copy $M^{(a)}$ of $M$, just as $M$ is definable in $V$. For this to be a bi-interpretation, we need to show that the isomorphism of $M$ with $M^{(a)}$ is definable in $M$. This isomorphism is precisely the function $i:x\mapsto x^{(a)}$, applied to $x\in M$. The problem again, however, is that $M$ is not able in general to undertake recursive definitions, since it satisfies only Zermelo set theory, which does not prove that recursive definitions have solutions. Nevertheless, we claim that $i\restrict x\in M$ for any $x\in M$. The reason is that the transitive closure of $i\restrict x=\set{(y,y^{(a)})\mid y\in x}$ can be seen to obey the growth-rate requirement, as $x$ does with its transitive closure and the sets $y^{(a)}$ are either $a$ or have no elements in any $V_n$, as in lemma \ref{Lemma.V^(a)-subset-M} above. Thus, $M$ can see the isomorphism $i\restrict x$ of $\<x,\in>$ with $\<x^{(a)},\in>$ for any set $x\in M$. If $x$ is transitive, then since transitive sets are rigid, this isomorphism is unique, and since different transitive sets are never isomorphic, $\<x,\in>$ will not be isomorphic to any $\<y,\in>$ that $M^{(a)}$ thinks is transitive. So $M$ can define $i\restrict x$ for transitive sets $x$ by the assertions that it is an isomorphism of $\<x,\in>$ with a set $\<y,\in>$ for a set $y$ that $M^{(a)}$ thinks is transitive. These maps all agree with and union up to the full isomorphism $x\mapsto x^{(a)}$ of $M$ with $M^{(a)}$, which is therefore definable in $M$. So we have a bi-interpretation of $\<V,\in>$ with $\<M,\in>$.
\end{proof}

Let us now put all this together and explain how we have proved theorem \ref{Theorem.Zermelo-not-solid-not-tight}.

\begin{proof}[Proof of theorem \ref{Theorem.Zermelo-not-solid-not-tight}]
The argument we have given shows that every model $V\satisfies\ZF$ is bi-interpretable with a transitive inner model $M$ of Zermelo set theory, which is not a model of \ZF, because it does not even have $V_\omega$ as an element. This establishes statement (2) of theorem \ref{Theorem.Zermelo-not-solid-not-tight}.

For statement (1), we can perform the construction inside a well-founded model of some sufficient fragment of \ZF. For example, let us undertake the entire construction inside some $V_\lambda$, satisfying $\Sigma_2$-collection, say. We thereby get a transitive model $M_\lambda\of V_\lambda$ of Zermelo set theory, without $V_\omega$, but with $\<V_\lambda,\in>$ and $\<M_\lambda,\in>$ bi-interpretable. So there are well-founded models of Zermelo set theory that are bi-interpretable, but not isomorphic.

Finally, let us consider statement (3). We consider the two extensions of Zermelo set theory describing the situation that enabled us to make the main example for statement (2). That is, the first theory is \ZF\ itself; and the second theory, which we denote ZM, asserts Zermelo set theory Z, plus the assertion that the class $V^{(\omega)}$, defined as in our main argument, is a model of \ZF, and the assertion that $M$ is isomorphic to $M^{(\omega)}$ inside that class by the isomorphism definition we provided. These two theories are different, but the argument that we gave above for the models shows that they are bi-interpretable.
\end{proof}

Let us close this section by mentioning that the construction is quite malleable with respect to the instance of replacement we want to eliminate. Alternative versions of the construction will construct supertransitive inner models $M$ that satisfy Zermelo set theory, and which have $V_\alpha$ as a set for every ordinal $\alpha<\lambda$, but do not have $V_\lambda$ as a set.

\begin{theorem}\label{Theorem.Zermelo-with-V_alphas}
For every limit ordinal $\lambda$, the universe $\<V,\in>$ is bi-interpretable with a transitive inner model $\<M,\in>$ satisfying Zermelo set theory with foundation, such that $V_\alpha\in M$ for every $\alpha<\lambda$, but $V_\lambda\notin M$.
\end{theorem}

\begin{proof}
We use $M=\Union T^{Q,G}$, where $Q:\lambda\to V_\lambda$ is $Q(\alpha)=V_\alpha$, and $G$ consists of the functions $f$ that are bounded by one of the functions $b_\beta$, for some $\beta<\lambda$, defined by
\begin{enumerate}
  \item $b_0(\alpha) = \alpha$;
  \item $b_{\beta+1}(\alpha) = 2^{b_\beta(\alpha)}$;
  \item if $\eta$ is a limit ordinal, then $b_\eta(\alpha) = \sup_{\beta<\eta}b_\beta(\alpha)$.
\end{enumerate}
This is defined with cardinal arithmetic, not ordinal arithmetic. The class $T^{Q,G}$ consists of all transitive sets $x$ whose rate-of-growth functions
$$f_x^Q(\alpha)=|x\intersect V_\alpha|$$
is bounded by one of the functions $b_\beta$, for some $\beta<\lambda$. It follows that $T^{Q,G}$ is a fruitful class, and the resulting Zermelo model $M=\Union T^{Q,G}$ is a transitive inner model of Zermelo set theory with foundation (in the form of the $\in$-induction scheme). Every $V_\alpha$ for $\alpha<\lambda$ is in $M$, with growth rate bounded by $b_\alpha$, but $V_\lambda$ is not in $M$, since this growth rate exceeds any given $b_\beta$ for fixed $\beta<\lambda$. And $\<V,\in>$ is bi-interpretable with $\<M,\in>$ by analogous arguments to the above.
\end{proof}

One can use this idea to produce many non-tight theories extending Zermelo set theory. For example, for any theory true in the model $M$ that we produced in the proof of theorem \ref{Theorem.Zermelo-not-solid-not-tight} will be bi-interpretable with \ZF. For example, that model $M$ satisfies the principle that every definable function from a set to the ordinals is bounded, which can be seen as a weak form of replacement not provable in Z, but true in M precisely because it is a transitive inner model (with all the ordinals) of a model of \ZF. In theorem \ref{Theorem.Zermelo-with-V_alphas} we similarly arranged that $V_\alpha$ exists for all $\alpha$ up to some large definable limit ordinal $\lambda$, which is itself another instance of replacement, and so this gives another extension of Zermelo set theory that is bi-interpretable with ZF.

Let us conclude the paper by distinguishing the theory form of bi-interpretation from a natural model-by-model form of bi-interpretation that might hold for the models of the theory. Specifically, we define that

\begin{definition}\rm
Theories $T_1$ and $T_2$ are \emph{model-by-model} bi-interpretable if every model of one of the theories is bi-interpretable with a model of the other.
\end{definition}
The definition in effect drops the uniformity requirement for bi-interpretability of theories, since it could be a different interpretation that works in one model than in another, with perhaps no uniform definition that works in all models of the theory. This is precisely what we should like now to prove.

\begin{theorem}
 There are theories $T_1$ and $T_2$ that are model-by-model bi-interpretable, but not bi-interpretable.
\end{theorem}

\begin{proof}
Consider the theories
\begin{enumerate}
  \item $T_1 = \ZF$.
  \item $T_2 = \{\alpha \lor \beta \mid \alpha \in \ZF \land \beta \in \ZM\}$.
  (That is, $T_2 = \ZM \lor \ZF$)
\end{enumerate}
where ZM is the theory used in the proof of theorem \ref{Theorem.Zermelo-not-solid-not-tight}. Observe that $M\satisfies T_2$ if and only if $M\satisfies\ZF$ or $M\satisfies\ZM$. The reverse implication is immediate, and if the latter fails, then $M\not\satisfies\alpha$ and $M\not\satisfies\beta$ for some $\alpha\in\ZF$ and $\beta\in\ZM$, which would mean $M\not\satisfies\alpha\lor\beta$, violating the former.

Now observe simply that every model of \ZF\ is bi-interpretable with itself, of course, and this is a model of $T_2$. And conversely, every model of $T_2$ is either a model of \ZF\ already, or else a model of \ZM, which is bi-interpretable with a model of \ZF. So these two theories are model-by-model bi-interpretable.

Let us now prove that the theories are not bi-interpretable. Suppose toward contradiction that they are, that $(I,J)$ forms a bi-interpretation, so that $I$ defines a model of $T_1$ inside any model of $T_2$ and $J$ defines a model of $T_2$ inside any model of $T_1$, and the theories prove this and furthermore have definable isomorphisms of the original models with their successive double-interpreted copies. Let $M$ be a model of $\ZM$ that is not a model of $\ZF$. Let $I^M$ be the interpreted model of \ZF\ defined inside $M$. This is also a model of \ZM, since $\ZM\of\ZF$, and so we may use $I$ again to define $I^{I^M}$ to get another model of \ZF\ inside that model. Because we interpreted a model of $T_2$ by $I$, it follows that $I^{I^M}$ and $I^M$ are bi-interpretable as models. But these are both models of \ZF, and so by theorem \ref{ali-theorem} they are isomorphic. By interpreting with $J$ to get the corresponding interpreted \ZM\ models, it follows that $J^{I^{I^M}}$ and $J^{I^M}$ are isomorphic. Because $(I,J)$ form a bi-interpretation, the first of these is isomorphic to $I^M$, which is a \ZF\ model, while the second is isomorphic to $M$, which is not. This is a contradiction, and so the theories are not bi-interpretable.
\end{proof}

The essence of the proof was that we had a theory \ZM\ that was not solid, but it was a subtheory of a solid theory \ZF, with which it was bi-interpretable.

\section{Final remarks}

We have investigated the nature of mutual and bi-interpretation in set theory and the models of set theory. There is surely a vibrant mutual interpretation phenomenon in set theory, with numerous instances of mutual interpretation amongst diverse natural extensions of \ZF\ set theory. This mutual-interpretation phenomenon, however, is less than fully robust in several respects. First, theories extending \ZF\ are never bi-interpretable, and indeed, there cannot be even a single nontrivial instance of models of \ZF\ set theory being bi-interpretable. Worse, amongst the well-founded models of \ZF, there cannot even be nontrivial instances of mutual interpretation. The basic lesson of this is that, when one follows the mutual interpretations we mentioned above, \emph{there is no getting back}. One cannot recover the original model. When you move from a well-founded model of \ZFC\ with large cardinals to the corresponding determinacy model, you cannot define a copy of the original model again, and the same when interpreting in the other direction. In this sense, interpretation in models of set theory inevitably involves and requires the loss of set-theoretic information.

Meanwhile, we have showed that several natural weak set theories avoid this phenomenon, including Zermelo-Fraenkel set theory $\ZFCm$ without the power set axiom and Zermelo set theory Z. It remains open just how strong the weak set theories can be, while still admitting non-trivial instances of bi-interpretation. In addition, there remain interesting open questions concerning the circumstances in which $H_{\omega_1}$ and $H_{\omega_2}$ might be bi-interpretable or bi-interpretation synonymous.


\bibliographystyle{halpha}
\bibliography{MathBiblio,HamkinsBiblio,WebPosts,references}

\end{document}

%% file: MathMacrosJDH.tex
\newtheorem{theorem}{Theorem}
\newtheorem*{theorem*}{Theorem}

\newtheorem*{maintheorem*}{Main Theorem}
\newtheorem*{maintheorems*}{Main Theorems}
\newtheorem{corollary}[theorem]{Corollary}
\newtheorem*{corollary*}{Corollary}
\newtheorem*{corollaries*}{Corollaries}

\newtheorem{lemma}[theorem]{Lemma}

\newtheorem{question}[theorem]{Question}
\newtheorem*{question*}{Question}

\newtheorem*{questions*}{Questions}
\newtheorem*{mainquestion*}{Main Question} 
\newtheorem*{openquestion*}{Open Question} 

\newtheorem{definition}[theorem]{Definition}

\newcommand{\QED}{\end{proof}}

\def\proclaim[#1]{{\bf #1}}
\def\BF#1.{{\bf #1.}}

\def\says#1:#2\par{\item[#1] #2\par}

\newcommand{\Los}{\L o\'s}

\newcommand{\Levy}{L\'{e}vy}

\newcommand\Vaananen{Väänänen}

\newcommand{\B}{{\mathbb B}}
\newcommand{\C}{{\mathbb C}}

\newcommand{\N}{{\mathbb N}}

\newcommand{\Q}{{\mathbb Q}}

\newcommand{\Z}{{\mathbb Z}}
\newcommand{\R}{{\mathbb R}}

\newcommand{\overbar}[1]{\mkern 3.5mu\overline{\mkern-3.5mu#1\mkern-.5mu}\mkern.5mu}
\newcommand{\barin}{\mkern3mu\overline{\mkern-3mu\in\mkern-1.5mu}\mkern1.5mu}

\newcommand{\Mbar}{{\overbar{M}}}
\newcommand{\Nbar}{{\overbar{N}}}

\makeatletter
\newcommand{\dotminus}{\mathbin{\text{\@dotminus}}}
\newcommand{\@dotminus}{%
  \ooalign{\hidewidth\raise1ex\hbox{.}\hidewidth\cr$\m@th-$\cr}%
}
\makeatother

\newcommand{\of}{\subseteq}

\newcommand{\set}[1]{\{\,{#1}\,\}}
\newcommand{\singleton}[1]{\left\{{#1}\right\}}

\newcommand{\Add}{\mathop{\rm Add}}

\newcommand{\restrict}{\upharpoonright} 
\newcommand{\satisfies}{\models}

\newcommand{\proves}{\vdash}

\newcommand{\union}{\cup}

\newcommand{\Union}{\bigcup}

\newcommand{\intersect}{\cap}

\newcommand{\smalllt}{\mathrel{\mathchoice{\raise2pt\hbox{$\scriptstyle<$}}{\raise1pt\hbox{$\scriptstyle<$}}{\raise0pt\hbox{$\scriptscriptstyle<$}}{\scriptscriptstyle<}}}
\newcommand{\smallleq}{\mathrel{\mathchoice{\raise2pt\hbox{$\scriptstyle\leq$}}{\raise1pt\hbox{$\scriptstyle\leq$}}{\raise1pt\hbox{$\scriptscriptstyle\leq$}}{\scriptscriptstyle\leq}}}

   \def\DHLhksqrt#1#2{%
   \setbox0=\hbox{$#1\sqrt{#2\,}$}\dimen0=\ht0
   \advance\dimen0-0.2\ht0
   \setbox2=\hbox{\vrule height\ht0 depth -\dimen0}%
   {\box0\lower0.4pt\box2}}
\newcommand{\boolval}[1]{\mathopen{\lbrack\!\lbrack}\,#1\,\mathclose{\rbrack\!\rbrack}}
\def\[#1]{\boolval{#1}}
\newbox\gnBoxA
\newdimen\gnCornerHgt
\setbox\gnBoxA=\hbox{\tiny$\ulcorner$}
\global\gnCornerHgt=\ht\gnBoxA
\newdimen\gnArgHgt
\def\gcode #1{%
\setbox\gnBoxA=\hbox{$#1$}%
\gnArgHgt=\ht\gnBoxA%
\ifnum     \gnArgHgt<\gnCornerHgt \gnArgHgt=0pt%
\else \advance \gnArgHgt by -\gnCornerHgt%
\fi \raise\gnArgHgt\hbox{\tiny$\ulcorner$} \box\gnBoxA %
\raise\gnArgHgt\hbox{\tiny$\urcorner$}}
\newcommand{\UnderTilde}[1]{{\setbox1=\hbox{$#1$}\baselineskip=0pt\vtop{\hbox{$#1$}\hbox to\wd1{\hfil$\sim$\hfil}}}{}}
\newcommand{\Undertilde}[1]{{\setbox1=\hbox{$#1$}\baselineskip=0pt\vtop{\hbox{$#1$}\hbox to\wd1{\hfil$\scriptstyle\sim$\hfil}}}{}}
\newcommand{\undertilde}[1]{{\setbox1=\hbox{$#1$}\baselineskip=0pt\vtop{\hbox{$#1$}\hbox to\wd1{\hfil$\scriptscriptstyle\sim$\hfil}}}{}}
\newcommand{\UnderdTilde}[1]{{\setbox1=\hbox{$#1$}\baselineskip=0pt\vtop{\hbox{$#1$}\hbox to\wd1{\hfil$\approx$\hfil}}}{}}
\newcommand{\Underdtilde}[1]{{\setbox1=\hbox{$#1$}\baselineskip=0pt\vtop{\hbox{$#1$}\hbox to\wd1{\hfil\scriptsize$\approx$\hfil}}}{}}

\renewcommand{\implies}{\mathrel{\rightarrow}}

\renewcommand{\iff}{\mathrel{\leftrightarrow}}
\newcommand{\Iff}{\mathrel{\Longleftrightarrow}}

\newcommand{\iso}{\cong}
\def\<#1>{\left\langle#1\right\rangle}

\newcommand{\TC}{\mathop{{\rm TC}}}

\newcommand{\Ord}{\mathord{{\rm Ord}}}

\newcommand{\ZFC}{{\rm ZFC}}
\newcommand{\ZF}{{\rm ZF}}

\newcommand{\ZFm}{\ZF^-}
\newcommand{\ZFCm}{\ZFC^-}

\newcommand{\KM}{{\rm KM}}

\newcommand{\CH}{{\rm CH}}

\newcommand{\GCH}{{\rm GCH}}

\newcommand{\AC}{{\rm AC}}

\newcommand{\AD}{{\rm AD}}

\newcommand{\MA}{{\rm MA}}

\newcommand{\HOD}{{\rm HOD}}

\newcommand{\HF}{{\rm HF}}

\newcommand{\PA}{{\rm PA}}

\newcommand{\cell}[1]{\boxit{\hbox to 17pt{\strut\hfil$#1$\hfil}}}
\newcommand{\head}[2]{\lower2pt\vbox{\hbox{\strut\footnotesize\it\hskip3pt#2}\boxit{\cell#1}}}
\newcommand{\boxit}[1]{\setbox4=\hbox{\kern2pt#1\kern2pt}\hbox{\vrule\vbox{\hrule\kern2pt\box4\kern2pt\hrule}\vrule}}
\newcommand{\Col}[3]{\hbox{\vbox{\baselineskip=0pt\parskip=0pt\cell#1\cell#2\cell#3}}}
\newcommand{\tapenames}{\raise 5pt\vbox to .7in{\hbox to .8in{\it\hfill input: \strut}\vfill\hbox to
.8in{\it\hfill scratch: \strut}\vfill\hbox to .8in{\it\hfill output: \strut}}}
\newcommand{\Head}[4]{\lower2pt\vbox{\hbox to25pt{\strut\footnotesize\it\hfill#4\hfill}\boxit{\Col#1#2#3}}}
\newcommand{\Dots}{\raise 5pt\vbox to .7in{\hbox{\ $\cdots$\strut}\vfill\hbox{\ $\cdots$\strut}\vfill\hbox{\
$\cdots$\strut}}}